\documentclass[11pt]{article}

\usepackage{amsmath,amssymb,amsfonts,textcomp,amsthm,xifthen,psfrag,graphicx,color,MnSymbol}
\usepackage{svg}
\usepackage[T1]{fontenc}
\usepackage{hyperref}
\usepackage{placeins}

\usepackage{pgfplots}
\pgfplotsset{compat=newest}
\usepgfplotslibrary{groupplots}
\usepgfplotslibrary{external}
\definecolor{navyblue}{rgb}{0.0, 0.0, 0.5}

\date{}

\newtheorem{theorem}{Theorem}
\newtheorem{lemma}[theorem]{Lemma}
\newtheorem{cor}[theorem]{Corollary}
\newtheorem{prop}[theorem]{Proposition}

\theoremstyle{definition} 

\newcommand{\<}{\langle{}}
\renewcommand{\>}{\rangle}

\newcommand{\ip}[2]{\llangle#1\hspace*{.5mm},#2\rrangle}
\newcommand{\dual}[2]{\<#1\hspace*{.5mm},#2\>}
\newcommand{\dualg}[2]{\<#1\hspace*{.5mm},#2\>^\nabla}
\newcommand{\duald}[2]{\<#1\hspace*{.5mm},#2\>^\mathrm{div}}
\newcommand{\dualgd}[2]{\<#1\hspace*{.5mm},#2\>^{\nabla\mathrm{div}}}

\newcommand{\vdual}[2]{(#1\hspace*{.5mm},#2)}

\newcommand{\diam}{\mathrm{diam}}

\newcommand{\wat}{\widehat}

\def\grad{\nabla}

\def\bz{\boldsymbol{z}}

\newcommand{\bL}{\ensuremath{\boldsymbol{L}}}

\def\tbu{\wat{\boldsymbol{u}}}

\def\dtu{{\delta\!\hat u}}
\def\tbw{\wat{\boldsymbol{w}}}
\def\tbv{\wat{\boldsymbol{v}}}
\def\tv{\hat{v}}
\def\tbv{\hat{\boldsymbol{v}}}
\def\ttau{\hat{\tau}}

\newcommand{\bg}{\boldsymbol{g}}

\newcommand{\bv}{\boldsymbol{v}}

\newcommand{\dbv}{\boldsymbol{\delta\!v}}
\newcommand{\bu}{\boldsymbol{u}}
\newcommand{\bw}{\boldsymbol{w}}

\newcommand{\dbu}{\boldsymbol{\delta\!u}}
\newcommand{\dtbu}{\boldsymbol{\delta\!\hat u}}

\newcommand{\uone}{{\boldsymbol{u}_1}}
\newcommand{\utwo}{{u_2}}
\newcommand{\uthree}{{\boldsymbol{u}_3}}
\newcommand{\ufour}{{u_4}}

\newcommand{\tuone}{{\hat u_1}}
\newcommand{\tutwo}{{\hat u_2}}
\newcommand{\tuthree}{{\hat u_3}}
\newcommand{\tufour}{{\hat u_4}}

\newcommand{\vone}{{\boldsymbol{v}_1}}
\newcommand{\vtwo}{{v_2}}
\newcommand{\vthree}{{\boldsymbol{v}_3}}
\newcommand{\vfour}{{v_4}}

\newcommand{\dvtwo}{{\delta\!v_2}}

\newcommand{\dvfour}{{\delta\!v_4}}

\newcommand{\gone}{{\boldsymbol{g}_1}}
\newcommand{\gtwo}{{g_2}}
\newcommand{\gthree}{{\boldsymbol{g}_3}}
\newcommand{\gfour}{{g_4}}

\newcommand{\uu}{\mathfrak{u}}
\newcommand{\duu}{\delta\!\mathfrak{u}}
\newcommand{\tuu}{\wat{\mathfrak{u}}}

\newcommand{\deltauu}{\delta\!\uu}
\newcommand{\vv}{\mathfrak{v}}
\newcommand{\ww}{\mathfrak{w}}

\newcommand{\UU}{\ensuremath{\mathfrak{U}}}
\newcommand{\tUU}{\ensuremath{\mathfrak{\wat U}}}
\newcommand{\VV}{\ensuremath{\mathfrak{V}}}

\newcommand{\bpsi}{\ensuremath{\boldsymbol{\psi}}}
\newcommand{\bphi}{\ensuremath{\boldsymbol{\phi}}}

\newcommand{\traceg}[1]{\mathrm{tr}_{#1}^{\mathrm{grad}}}
\newcommand{\traced}[1]{\mathrm{tr}_{#1}^{\mathrm{div}}}
\newcommand{\tracegd}[1]{\mathrm{tr}_{#1}^{\grad\mathrm{div}}}

\newcommand{\Tref}{{\wat T}}
\newcommand{\pPFgd}[1]{{\widetilde\Pi^{\grad\mathrm{div}}_{#1}}}
\newcommand{\PFgd}[1]{{\Pi^{\grad\mathrm{div}}_{#1}}}

\newcommand{\Hdiv}[1]{{\bH(\div\!,#1)}}
\newcommand{\Hcurl}[1]{{\bH(\bcurl,#1)}}
\newcommand{\Hdivz}[1]{{\bH_0(\div\!,#1)}}
\newcommand{\Hgdiv}[1]{{\bH(\grad\div\!,#1)}}
\newcommand{\Hgdivz}[1]{{\bH_0(\grad\div\!,#1)}}
\newcommand{\trHgd}[1]{{\bH^{\grad\div}(#1)}}
\newcommand{\trHgdz}[1]{{\bH_{0}^{\grad\div}(#1)}}

\newcommand{\bH}{\ensuremath{\boldsymbol{H}}}

\def\curl{{\rm curl\,}}

\DeclareMathOperator{\bcurl}{{\bf curl}}

\def\div{{\rm div\,}}

\newcommand{\ttt}{{\mathfrak{T}}}

\newcommand{\R}{\ensuremath{\mathbb{R}}}
\newcommand{\N}{\ensuremath{\mathbb{N}}}

\newcommand{\cI}{\ensuremath{\mathcal{I}}}
\newcommand{\nn}{\ensuremath{\boldsymbol{n}}}
\newcommand{\ff}{\boldsymbol{f}}

\newcommand{\cT}{\ensuremath{\mathcal{T}}}

\newcommand{\cS}{\ensuremath{\mathcal{S}}}

\newcommand{\cP}{\ensuremath{\mathcal{P}}}

\newcommand{\cE}{\ensuremath{\mathcal{E}}}
\newcommand{\cF}{\ensuremath{\mathcal{F}}}

\newcommand{\btau}{{\boldsymbol\tau}}
\newcommand{\dbtau}{{\boldsymbol{\delta\!\tau}}}

\title{DPG methods for a fourth-order div problem
\thanks{Supported by ANID-Chile through FONDECYT projects 1190009, 1210391}
\author{
Thomas~F\"uhrer$^\dagger$
\and
Pablo Herrera$^\dagger$
\and
Norbert Heuer\thanks{
Facultad de Matem\'aticas, Pontificia Universidad Cat\'olica de Chile,
Avenida Vicu\~na Mackenna 4860, Santiago, Chile,
email: {\tt \{tofuhrer,pcherrera,nheuer\}@mat.uc.cl}}}}

\begin{document}
\maketitle
\begin{abstract}
We study a fourth-order div problem and its approximation by
the discontinuous Petrov--Galerkin method with optimal test functions.
We present two variants, based on first and second-order systems. In both cases
we prove well-posedness of the formulation and quasi-optimal convergence of the
approximation. Our analysis includes the fully-discrete schemes with approximated test functions,
for general dimension and polynomial degree in the first-order case, and for two dimensions
and lowest-order approximation in the second-order case. Numerical results illustrate
the performance for quasi-uniform and adaptively refined meshes.


\end{abstract}

\section{Introduction}

There has been a sustained interest in the numerical analysis of fourth-order problems. Important
examples are some models for the mechanics of thin structures
\cite{Fluegge_60_SS,Kirchhoff_50_UGB,Love_88_SFV}, and more recently the quad-curl problem
related with Maxwell transmission eigenvalue problems \cite{CakoniCMS_10_IES,MonkS_12_FEM} and
magneto-hydrodynamics with hyper-resistivity \cite{Biskamp_00_MRP,ChaconSZ_07_SSP}.
In this paper, we study a fourth-order problem with the $(\grad\div)^2$-operator.
In two dimensions this operator is nothing but the quad-curl operator whose approximation
has been studied, e.g., in \cite{BrennerSS_17_HDM,HongHSX_12_DGM,Sun_16_MFQ}, to give a few references.
In \cite{FanLZ_19_MSF}, a mixed method for the very $(\grad\div)^2$-operator is considered.

The motivation of this work is to continue developing DPG techniques for fourth-order problems.
The DPG framework, short for discontinuous Petrov--Galerkin method with optimal test functions,
has been established by Demkowicz and Gopalakrishnan \cite{DemkowiczG_11_CDP}
to aim at automatic discrete stability of (Petrov) Galerkin schemes
for well-posed variational formulations. This is particularly relevant for
fluid flow and singularly perturbed problems, cf., e.g.,
\cite{ChanHBTD_14_RDM,DemkowiczH_13_RDM,FuehrerH_17_RCD,HeuerK_17_RDM,RobertsDM_15_DPG}.
The DPG approach has important advantages over traditional and mixed Galerkin settings.
First, any well-posed variational formulation is suitable. Second, this includes ultraweak formulations
where conforming approximations are easier to realize since only trace variables are affected.
Third, being a minimum residual method, DPG has built-in a posteriori error estimation and adaptivity,
cf.~\cite{CarstensenDG_14_PEC,DemkowiczFHT_21_DAP,DemkowiczGN_12_CDP}.

The development of stable formulations and characterization of traces is at the heart of
our studies of DPG formulations for thin structures,
\cite{FuehrerHH_21_TOB,FuehrerHN_19_UFK,FuehrerHN_DMS,FuehrerHS_20_UFR}.
Here we continue the DPG development for higher-order problems, considering a fourth-order div problem.
We present two variational formulations of ultraweak type based on systems of first and second order,
and analyze their approximations by the DPG method. As we will see, the analysis of the
first-order case is relatively straightforward, though including that of a saddle
point system with non-zero diagonal blocks. On the other hand, the second-order
formulation reduces the number of unknowns (though static condensation of field
variables is an option) but complicates the analysis of traces typically
appearing in ultraweak formulations.
This is particularly the case when analyzing the fully discrete setting that includes
the approximation of optimal test functions, cf.~\cite{GopalakrishnanQ_14_APD}.
In the end, the approximation of traces is identical in both settings.

Let us finish the introduction with presenting our model problem and giving a description of what follows.

Let $\Omega\subset\R^d$ ($d\ge 2$) be a bounded, simply connected Lipschitz domain
with boundary $\Gamma=\partial\Omega$.
For a given vector function $\ff$ we consider the problem of finding the vector function $\bu$
that satisfies
\begin{align} \label{prob_org}
   (\grad\div\!)^2\bu + \bu = \ff\quad\text{in}\ \Omega,\quad
   \div\bu = \bu\cdot\nn = 0\quad\text{on}\ \Gamma.
\end{align}
Here, $\nn$ indicates the exterior unit normal vector on $\Gamma$, and below will be used
generically along the boundary of any Lipschitz domain.
For ease of presentation we have selected homogeneous boundary conditions. This is not essential
and in fact, one of our numerical examples requires non-homogeneous traces on a part of the
boundary.

In the next section we recall and study trace operators and spaces
that stem from the $(\grad\div)^2$-operator.
For the first-order system these are classical operators, recalled in \S\ref{sec_traces_canonical},
whereas in \S\ref{sec_traces_gd} we introduce operators needed for the second-order system.
The first-order formulation, along with the general DPG method and a specific discrete setting,
is introduced and analyzed in Section~\ref{sec_DPG1}. Well-posedness of the formulation
and quasi-optimal convergence of the DPG method with optimal test functions and with
approximated test functions are stated in Theorems~\ref{thm_stab1},\ref{thm_DPG1}
and Corollary~\ref{cor_DPG1_discrete}. Here, the fully discrete case considers a general
space dimension and arbitrary polynomial degrees (though uniform for simplicity),
see~\S\ref{sec_discrete1}. Proofs are collected in \S\ref{sec_pf1}.
Section~\ref{sec_DPG2} considers the second-order system,
using the same setup as in Section~\ref{sec_DPG1}. Well-posedness of the formulation
and quasi-optimal convergence of the schemes are stated by Theorems~\ref{thm_stab2},\ref{thm_DPG2}
and Corollary~\ref{cor_DPG2_discrete}. The main difference in this case is that the fully
discrete setting is only analyzed for two space dimensions and lowest-order approximation,
see~\S\ref{sec_discrete2}. Proofs for the setup of the second-order system are collected
in \S\ref{sec_pf2}. In Section~\ref{sec_num} we present numerical results
for both DPG schemes, considering an example with smooth solution and an example with corner
singularity. In the latter case we also consider adaptive variants of both schemes.
In all cases we observe the expected convergence orders.

Throughout the paper, $a\lesssim b$ means that $a\le cb$ with a generic constant $c>0$
that is independent of involved functions and diameters of elements, except otherwise noted.
Similarly, we use the notation $a\gtrsim b$, and $a\simeq b$ means that $a\lesssim b$ and
$b\lesssim a$.

\section{Trace spaces and norms} \label{sec_traces}

We consider a mesh $\cT=\{T\}$ consisting of general non-intersecting Lipschitz elements
so that $\bar\Omega=\cup\{\bar T;\; T\in\cT\}$, and denote by
$\cS=\{\partial T;\;T\in\cT\}$ its skeleton.

Throughout this section, let $\omega\subset\Omega$ denote a Lipschitz
domain, for simplicity simply connected, with boundary $\partial\omega$.
We use the standard $L_2$ spaces
$L_2(\omega)$, $\bL_2(\omega)=(L_2(\omega))^d$ of scalar and vector fields, respectively,
with generic norm $\|\cdot\|_\omega$ and $L_2(\omega)$-bilinear form $\vdual{\cdot}{\cdot}_\omega$.
We drop the index $\omega$ when $\omega=\Omega$. We also use the standard Sobolev spaces
$H^1(\omega)$ and $H^1_0(\omega)$ with norm $\|\cdot\|_{1,\omega}$. Furthermore,
\begin{align*}
   &\Hdiv{\omega} := \{\bv\in\bL_2(\omega);\; \div\bv\in L_2(\omega)\},\quad
   \Hdivz{\omega} := \{\bv\in\Hdiv{\omega};\; (\bv\cdot\nn)|_{\partial\omega}=0\},\\
   &\Hgdiv{\omega} := \{\bv\in\bL_2(\omega);\; \grad\div\bv\in \bL_2(\omega)\},\\
   &\Hgdivz{\omega} := \{\bv\in\Hgdiv{\omega};\;
                         (\bv\cdot\nn)|_{\partial\omega}=(\div\bv)|_{\partial\omega}=0\}
\end{align*}
with respective norms
\[
   \|\bv\|_{\div\!,\omega}^2 := \|\bv\|_\omega^2 + \|\div\bv\|_\omega^2,\quad
   \|\bv\|_{\grad\div\!,\omega}^2 := \|\bv\|_\omega^2 + \|\grad\div\bv\|_\omega^2,
\]
and dropping $\omega$ when $\omega=\Omega$.
We also need the corresponding product spaces, denoted in the same way but replacing $\omega$
with $\cT$, e.g., $\Hgdiv{\cT}=\Pi_{T\in\cT} \Hgdiv{T}$ with canonical product norm
$\|\cdot\|_{\grad\div\!,\cT}$. The $L_2(\cT)$-dualities are denoted by $\vdual{\cdot}{\cdot}_\cT$.

\subsection{Canonical trace operators} \label{sec_traces_canonical}

We define operators
\begin{align*} 
   \traceg{}:\;
   \begin{cases}
      H^1(\cT) &\to\ \Hdiv{\cT}',\\
      \quad v    &\mapsto\ \dualg{\traceg{}(v)}{\btau}_\cS
                    := \vdual{v}{\div\btau}_\cT + \vdual{\grad v}{\btau}_\cT
   \end{cases}
\end{align*}
and
\begin{align*} 
   \traced{}:\;
   \begin{cases}
      \Hdiv{\cT} &\to H^1(\cT)' \\
      \quad \btau    &\mapsto\ \duald{\traced{}(\btau)}{v}_\cS
                    := \dualg{\traceg{}(v)}{\btau}_\cS.
   \end{cases}
\end{align*}
Analogously we define the canonical trace operators for $\omega\subset\Omega$
(instead of $\cT$) with support on $\partial\omega$,
$\traceg{\omega}:\;H^1(\omega)\to\Hdiv{\omega}'$ and
$\traced{\omega}:\;\Hdiv{\omega}\to H^1(\omega)'$. Of course, given
$v\in H^1(\omega)$ and $\btau\in\Hdiv{\omega}$,
$\traceg{\omega}(v)=v|_{\partial\omega}$ and
$\traced{\omega}(\btau)=(\btau\cdot\nn)|_{\partial\omega}$
are the classical traces living in the canonical Sobolev spaces
$H^{1/2}(\partial\omega)=\traceg{\omega}(H^1(\omega))$ and
$H^{-1/2}(\partial\omega)=\traced{\omega}(\Hdiv{\omega})$, respectively.

The statements of the following proposition are proved in
\cite[Theorem 2.3, Remark 2.5]{CarstensenDG_16_BSF}.

\begin{prop} \label{prop_can_jumps}
(i) If $v\in H^1(\cT)$, then
\[
   v\in H^1_0(\Omega) \quad\Leftrightarrow\quad
   \duald{\traced{}(\btau)}{v}_\cS = 0
   \quad\forall \btau\in \Hdiv{\Omega}
\]
and
\[
   v\in H^1(\Omega) \quad\Leftrightarrow\quad
   \duald{\traced{}(\btau)}{v}_\cS = 0
   \quad\forall \btau\in \Hdivz{\Omega}.
\]
(ii) If $\btau\in \Hdiv{\cT}$, then
\[
   \btau\in \Hdivz{\Omega} \quad\Leftrightarrow\quad
   \dualg{\traceg{}(v)}{\btau}_\cS = 0
   \quad\forall v\in H^1(\Omega)
\]
and
\[
   \btau\in \Hdiv{\Omega} \quad\Leftrightarrow\quad
   \dualg{\traceg{}(v)}{\btau}_\cS = 0
   \quad\forall v\in H^1_0(\Omega).
\]
\end{prop}

The canonical trace spaces are
\begin{align*}
   H^{1/2}(\partial\omega) &:= \traceg{\omega}(H^1(\omega)),\qquad
   H^{1/2}(\cS) := \traceg{}(H^1(\Omega)),\qquad
   H^{1/2}_{0}(\cS) := \traceg{}(H^1_0(\Omega)),
   \\
   H^{-1/2}(\partial\omega) &:= \traced{\omega}(\Hdiv{\omega}),\quad
   H^{-1/2}(\cS) := \traced{}(\Hdiv{\Omega}),\quad
   H^{-1/2}_{0}(\cS) := \traced{}(\Hdivz{\Omega})
\end{align*}
with respective operator norms
\begin{align*}
   \|\tv\|_{(\div\!,\omega)'}
   &:= \sup_{\|\btau\|_{\div\!,\omega}=1} \dualg{\tv}{\btau}_{\partial\omega}
   &&\hspace{-8em} (\tv\in H^{1/2}(\partial\omega)),\\
   \|\tv\|_{(\div\!,\cT)'} &:= \sup_{\|\btau\|_{\div\!,\cT}=1} \dualg{\tv}{\btau}_\cS
   &&\hspace{-8em} (\tv\in H^{1/2}(\cS)),\\
   \|\ttau\|_{(1,\omega)'} &:= \sup_{\|v\|_{1,\omega}=1} \duald{\ttau}{v}_{\partial\omega}
   &&\hspace{-8em} (\ttau\in H^{-1/2}(\partial\omega)),\\
   \|\ttau\|_{(1,\cT)'} &:= \sup_{\|v\|_{1,\cT}=1} \duald{\ttau}{v}_\cS
   &&\hspace{-8em} (\ttau\in H^{-1/2}(\cS)).
\end{align*}
Here, the dualities between the trace spaces and their corresponding test spaces
are taken to be consistent with the trace operator.
For instance, given $\tv\in H^{1/2}(\cS)$ and $\btau\in\Hdiv{\cT}$,
\[
   \dualg{\tv}{\btau}_\cS := \dualg{\traceg{}(v)}{\btau}_\cS
   \quad\text{with}\ v\in H^1(\Omega)\ \text{s.th.}\ \traceg{}(v)=\tv.
\]
Alternatively, the canonical trace norms are
\begin{align*}
   \|\tv\|_{1/2,\partial\omega}
   &:= \inf\,\{\|v\|_{1,\omega};\; v\in H^1(\omega),\ \traceg{\omega}(v)=\tv\}
   &&\hspace{-3em} (\tv\in H^{1/2}(\partial\omega)),\\
   \|\tv\|_{1/2,\cS} &:= \inf\,\{\|v\|_1;\; v\in H^1(\Omega),\ \traceg{}(v)=\tv\}
   &&\hspace{-3em} (\tv\in H^{1/2}(\cS)),\\
   \|\ttau\|_{-1/2,\partial\omega}
   &:= \inf\,\{\|\btau\|_{\div\!,\omega};\; \btau\in \Hdiv{\omega},\ \traced{\omega}(\btau)=\ttau\}
   &&\hspace{-3em} (\ttau\in H^{-1/2}(\partial\omega)),\\
   \|\ttau\|_{-1/2,\cS}
   &:= \inf\,\{\|\btau\|_{\div};\; \btau\in \Hdiv{\Omega},\ \traced{}(\btau)=\ttau\}
   &&\hspace{-3em} (\ttau\in H^{-1/2}(\cS)).
\end{align*}
The equivalence of the corresponding trace norms will be relevant.
In fact, they are pairwise identical as stated by the next proposition.

\begin{prop}\cite[Lemma~2.1]{CarstensenDG_16_BSF} \label{prop_can_norms}
The identities
\begin{align*}
   \|\tv\|_{1/2,\partial\omega} &= \|\tv\|_{(\div\!,\omega)'}
   &&\hspace{-8em} (\tv\in H^{1/2}(\partial\omega)),\\
   \|\tv\|_{1/2,\cS} &= \|\tv\|_{(\div\!,\cT)'}
   &&\hspace{-8em} (\tv\in H^{1/2}(\cS)),\\
   \|\ttau\|_{-1/2,\partial\omega} &= \|\ttau\|_{(1,\omega)'}
   &&\hspace{-8em} (\ttau\in H^{-1/2}(\partial\omega)),\\
   \|\ttau\|_{-1/2,\cS} &= \|\ttau\|_{(1,\cT)'}
   &&\hspace{-8em} (\ttau\in H^{-1/2}(\cS))
\end{align*}
hold true.
\end{prop}

\subsection{A grad-div trace operator} \label{sec_traces_gd}

\begin{prop} \label{prop_gd} Let $\omega\subset\Omega$ be a Lipschitz domain. The norms
$\|\cdot\|_{\grad\div\!,\omega}$ and
\(
   \|\div\cdot\|_\omega + \|\cdot\|_{\grad\div\!,\omega}
\)
are equivalent in $\Hgdiv{\omega}$, with the non-trivial constant depending on $\omega$.
\end{prop}

\begin{proof}
Noting that $\grad\div\bv\in\bL_2(\omega)$ implies that $\div\bv\in L_2(\Omega)$
it follows that $\Hgdiv{\omega}$ is Banach with respect to both norms,
$\|\cdot\|_{\grad\div\!,\omega}$ and $\|\cdot\|_{\grad\div\!,\omega}+\|\div\cdot\|_\omega$.
The trivial bound
$\|\cdot\|_{\grad\div\!,\omega}\le \|\cdot\|_{\grad\div\!,\omega}+\|\div\cdot\|_\omega$
then implies the equivalence of both norms.
\end{proof}

Analogously as the canonical traces, we define traces induced by the operator $\grad\div$,
\begin{align*} 
   \tracegd{}:\;
   \begin{cases}
      \Hgdiv{\cT} &\to\ \Hgdiv{\cT}',\\
      \qquad \bv    &\mapsto\ \dualgd{\tracegd{}(\bv)}{\btau}_\cS
                    := \vdual{\bv}{\grad\div\btau}_\cT - \vdual{\grad\div\bv}{\btau}_\cT,
   \end{cases}
\end{align*}
and analogously for $\omega\subset\Omega$ (instead of $\cT$)
$\tracegd{\omega}:\;\Hgdiv{\omega}\to\Hgdiv{\omega}'$
with support on $\partial\omega$,
and duality $\dualgd{\cdot}{\cdot}_{\partial\omega}$.
By Proposition~\ref{prop_gd}, $\bv\in\Hgdiv{\omega}$ implies that
$\bv\in\Hdiv{\omega}$ and $\div\bv\in H^1(\omega)$. Therefore,
$\tracegd{\omega}(\bv)$ has the two well-posed classical trace components
$\bv\cdot\nn|_{\partial\omega}$ and $(\div\bv)|_{\partial\omega}$, and we can write
\begin{align} \label{trace_rel}
   \dualgd{\tracegd{\omega}(\bv)}{\btau}_{\partial\omega}
   =
   \duald{\traced{\omega}(\bv)}{\div\btau}_{\partial\omega}
   -
   \dualg{\traceg{\omega}(\div\bv)}{\btau}_{\partial\omega}
   \quad\forall \bv,\btau\in\Hgdiv{\omega}.
\end{align}
Though, the separate canonical components are not uniformly bounded with respect to the domain
$\omega$ under consideration.
In the following, $\tracegd{\omega}(\bv)=(g_1,g_2)$ for $\bv\in\Hgdiv{\omega}$
means that $g_1=\traced{\omega}(\bv)=(\bv\cdot\nn)|_{\partial\omega}$ and
$g_2=\traceg{\omega}(\div\bv)=(\div\bv)|_{\partial\omega}$.

We proceed with considering the trace operator $\tracegd{}$
and its induced trace spaces.
We define local trace spaces
\begin{align*}
   \trHgd{\partial\omega} := \tracegd{\omega}(\Hgdiv{\omega})
\end{align*}
and the ``skeleton'' variants
\begin{align*}
   \trHgd{\cS} &:= \tracegd{}(\Hgdiv{\Omega}),\quad
   \trHgdz{\cS} := \tracegd{}(\Hgdivz{\Omega})
\end{align*}
provided with operator norms,
\begin{align*}
   \|\tbv\|_{(\grad\div\!,\omega)'}
   &:= \sup_{\|\btau\|_{\grad\div\!,\omega}=1} \dualgd{\tbv}{\btau}_{\partial\omega}
   \quad (\tbv\in \trHgd{\partial\omega}),\\
   \|\tbv\|_{(\grad\div\!,\cT)'} &:= \sup_{\|\btau\|_{\grad\div\!,\cT}=1} \dualgd{\tbv}{\btau}_\cS
   \quad (\tbv\in \trHgd{\cS}).
\end{align*}
As before, the dualities are defined to be consistent with the definition of the trace operator.
The canonical trace norms are
\begin{align*}
   &\|\tbv\|_{-1/2,1/2,\partial\omega}
   := \inf\,\{\|\bv\|_{\grad\div\!,\omega};\; \bv\in \Hgdiv{\omega},\ \tracegd{\omega}(\bv)=\tbv\}
   \quad (\tbv\in\trHgd{\partial\omega}),\\
   &\|\tbv\|_{-1/2,1/2,\cS}
   := \inf\,\{\|\bv\|_{\grad\div};\; \bv\in \Hgdiv{\Omega},\ \tracegd{}(\bv)=\tbv\}
   \qquad (\tbv\in\trHgd{\cS}),
\end{align*}
and in Proposition~\ref{prop_gd_norm} we will see that
$\|\cdot\|_{-1/2,1/2,\cS}=\|\cdot\|_{(\grad\div\!,\cT)'}$.

The following proposition relates the combined $\grad\div$ traces with the canonical
ones. We formulate it for an arbitrary Lipschitz domain but only will use it for $\Omega$.

\begin{prop} \label{prop_gd_trace} Let $\omega\subset\Omega$ be a Lipschitz domain.
The spaces
\[
    \bigl(\trHgd{\partial\omega}, \|\cdot\|_{-1/2,1/2,\partial\omega}\bigr),\quad
    \bigl(H^{-1/2}(\partial\omega), \|\cdot\|_{-1/2,\partial\omega}\bigr) \times
    \bigl(H^{1/2}(\partial\omega), \|\cdot\|_{1/2,\partial\omega}\bigr)
\]
are identical with equivalent norms and equivalence constants that may depend on $\omega$.
More specifically, the inclusion
$H^{-1/2}(\partial\omega)\times H^{1/2}(\partial\omega)\subset \trHgd{\partial\omega}$
is due to the existence of a bounded extension operator
$\cE_\omega:\; H^{-1/2}(\partial\omega)\times H^{1/2}(\partial\omega)\to \Hgdiv{\omega}$
with $\tracegd{\omega}\circ\cE_\omega=id$.
\end{prop}

\begin{proof}
We already have seen trace identification \eqref{trace_rel} which shows the algebraic
inclusion $\trHgd{\partial\omega}\subset H^{-1/2}(\partial\omega)\times H^{1/2}(\partial\omega)$.

Now, to show the other inclusion and its continuity let
$(g_1,g_2)\in H^{-1/2}(\partial\omega)\times H^{1/2}(\partial\omega)$ be given.
We extend $\bg_1:=(g_1,0)\mapsto \bv_1\in\Hgdiv{\omega}$ and
$\bg_2:=(0,g_2)\mapsto \bv_2\in\Hgdiv{\omega}$ such that
$\tracegd{\omega}(\bv_j)=(\traced{\omega}(\bv_j),\traceg{\omega}(\div\bv_j))=\bg_j$,
and with boundedness
$\|\bv_1\|_{\grad\div\!,\omega}\lesssim \|g_1\|_{-1/2,\partial\omega}$
and
$\|\bv_2\|_{\grad\div\!,\omega}\lesssim \|g_2\|_{1/2,\partial\omega}$ ($j=1,2$).
The extension $(g_1,g_2)\mapsto \cE_\omega(g_1,g_2):= \bv:=\bv_1+\bv_2$ then satisfies
$\bv\in\Hgdiv{\omega}$ with $\tracegd{\omega}(\bv)=(g_1,g_2)$ and the bound
\[
   \|\tracegd{\omega}(\bv)\|_{-1/2,1/2,\partial\omega}
   \le
   \|\bv\|_{\grad\div\!,\omega}
   \lesssim \|g_1\|_{-1/2,\partial\omega} + \|g_2\|_{1/2,\partial\omega}.
\]
The equivalence of both norms then follows since
$\trHgd{\partial\omega}=H^{-1/2}(\partial\omega)\times H^{1/2}(\partial\omega)$ are Banach spaces.
\\
{\bf Extension of $\bg_1$.}
We define $\phi\in H^1_0(\omega)$ and $\psi\in H^1(\omega)$ as the solutions to
\begin{align*}
   &-\Delta\phi=\alpha\quad\text{in}\ \omega,\\
   &-\Delta\psi=\phi  \quad\text{in}\ \omega,\quad
   \partial_{\nn} \psi=-g_1\quad\text{on}\ \partial\omega,
\end{align*}
where $\alpha\in\R$ is chosen so that $\vdual{\phi}{1}_\omega-\duald{g_1}{1}_{\partial\omega}=0$.
Then we select $\bv_1:=-\grad\psi$ and observe that $\curl\bv_1=(\grad\div)^2\bv_1=0$
and $\tracegd{\omega}(\bv_1)=(\bv_1\cdot\nn,\div\bv_1)|_{\partial\omega}=(g_1,0)$. Furthermore,
\begin{align*}
   &\|\bv_1\|_\omega^2 = \|\grad\psi\|_\omega^2
   \lesssim \|\phi\|_\omega^2 + \|g_1\|_{-1/2,\partial\omega}^2,\\
   &\|\grad\div\bv_1\|_\omega^2 + \|\phi\|_\omega^2 = \|\grad\phi\|_\omega^2  + \|\phi\|_\omega^2
   \lesssim \|\grad\phi\|_\omega^2.
\end{align*}
It is not difficult to see that $\|\grad\phi\|_\omega\simeq |\alpha|\lesssim \|g_1\|_{-1/2,\partial\omega}$.
This finishes the extension of $\bg_1$.\\
{\bf Extension of $\bg_2$.}
We define $\phi\in H^1(\omega)$ and $\psi\in H^1(\omega)$ as the solutions to
\begin{align*}
   &-\Delta\phi=0\quad\text{in}\ \omega,\quad
   \phi=g_2\quad\text{on}\ \partial\omega,\\
   &-\Delta\psi=\phi  \quad\text{in}\ \omega,\quad
   \partial_{\nn} \psi=\beta\quad\text{on}\ \partial\omega,
\end{align*}
where $\beta\in\R$ is chosen so that $\vdual{\phi}{1}_\omega+\duald{\beta}{1}_{\partial\omega}=0$.
Let us denote the extension $\bg_1=(g_1,0)\mapsto\bv_1$ from the previous step as $\cE^1_\omega(g_1)$.
Then we define $\bv_2:=\cE^1_\omega(\beta)-\grad\psi$ and observe that
$\curl\bv_2=(\grad\div)^2\bv_2=0$
and $\tracegd{\omega}(\bv_2)=(0,g_2)$. Furthermore,
\begin{align*}
   &\|\bv_2\|_\omega^2 \le 2(\|\cE^1_\omega(\beta)\|_\omega^2 + \|\grad\psi\|_\omega^2)
   \lesssim |\beta|^2 + \|\phi\|_\omega^2 \lesssim |\beta|^2 + \|g_2\|_{1/2,\partial\omega}^2,\\
   &\|\grad\div\bv_2\|_\omega^2
   \le 2(\|\grad\div\cE^1_\omega(\beta)\|_\omega^2 + \|\grad\phi\|_\omega^2)
   \lesssim |\beta|^2 + \|g_2\|_{1/2,\partial\omega}^2.
\end{align*}
The bound $|\beta|\lesssim \|\phi\|_\omega\lesssim \|g_2\|_{1/2,\partial\omega}$ finally
proves the required boundedness of the extension $\bv_2$.
\end{proof}

The following propositions will be needed to analyze our second-order ultraweak formulation of
problem \eqref{prob_org}.

\begin{prop} \label{prop_gd_norm}
The identity
\begin{align*}
   \|\tbv\|_{-1/2,1/2,\cS} &= \|\tbv\|_{(\grad\div\!,\cT)'}\quad (\tbv\in \trHgd{\cS})
\end{align*}
holds true.
\end{prop}

\begin{proof}
The proof of this result is by now standard,
cf.~\cite[Proofs of Lemma~2.2, Theorem~2.3]{CarstensenDG_16_BSF}
and \cite[Proofs of Lemma~3.2, Proposition 3.5]{FuehrerHN_19_UFK}. We only recall the essential
steps, following the setting from \cite{FuehrerHN_19_UFK}. The bound
$\|\cdot\|_{(\grad\div\!,\cT)'}\le\|\cdot\|_{-1/2,1/2,\cS}$ is immediate by definition
of the trace operator. For the other inequality let $\tbv\in \trHgd{\cS}$ be given.
We define $\btau\in \Hgdiv{\cT}$ as the solution to
\[
   \vdual{\btau}{\dbtau} + \vdual{\grad\div\btau}{\grad\div\dbtau}_\cT
   = \dualgd{\tbv}{\dbtau}_\cS\quad\forall\dbtau\in\Hgdiv{\cT}
\]
and $\btau\mapsto\bv\in\Hgdiv{\Omega}$ as the solution to
\[
   \vdual{\bv}{\dbv} + \vdual{\grad\div\bv}{\grad\div\dbv}
   = \dualgd{\tracegd{}(\dbv)}{\btau}_\cS\quad\forall\dbv\in\Hgdiv{\Omega}.
\]
It follows that $\bv=\grad\div\btau$ (piecewise on $\cT$),
leading to $\tracegd{}(\bv)=\tbv$ and
$\|\btau\|_{\grad\div\!,\cT}^2=\|\bv\|_{\grad\div}^2=\dualgd{\tbv}{\btau}_\cS$, so that
\[
   \|\tbv\|_{(\grad\div\!,\cT)'}
   \ge \frac {\dualgd{\tbv}{\btau}_\cS}{\|\btau\|_{\grad\div\!,\cT}}
   = \|\bv\|_{\grad\div} \ge \|\tbv\|_{-1/2,1/2,\cS}.
\]
\end{proof}

\begin{prop} \label{prop_gd_jumps}
If $\bv\in \Hgdiv{\cT}$, then
\[
   \bv\in \Hgdivz{\Omega} \quad\Leftrightarrow\quad
   \dualgd{\tracegd{}(\btau)}{\bv}_\cS = 0
   \quad\forall \btau\in \Hgdiv{\Omega}
\]
and
\[
   \bv\in \Hgdiv{\Omega} \quad\Leftrightarrow\quad
   \dualgd{\tracegd{}(\btau)}{\bv}_\cS = 0
   \quad\forall \btau\in \Hgdivz{\Omega}.
\]
\end{prop}

\begin{proof}
The statements can be proved by standard techniques. Let us recall them briefly.
In both cases the direction ``$\Rightarrow$'' holds by integration by parts, taking
into account the inherent regularities of $\Hgdiv{\Omega}$ mentioned previously.
The direction ``$\Leftarrow$'' without boundary condition holds by the distributional
definition of the derivatives:
\begin{align*}
   \grad\div\bv(\bphi) := \vdual{\bv}{\grad\div\bphi}
   = \vdual{\grad\div\bv}{\bphi}_\cT + \dualgd{\tracegd{}(\bv)}{\bphi}_\cS
   = \vdual{\grad\div\bv}{\bphi}_\cT
   \quad\forall\bphi\in C_0^\infty(\Omega)^d,
\end{align*}
giving $\grad\div\bv\in \bL_2(\Omega)$. Now, to show the boundary condition in the first statement,
let $\bv\in\Hgdiv{\Omega}$ be given. We use Proposition~\ref{prop_gd_trace} to extend an arbitrary
datum $\bg=(g_1,g_2)\in H^{-1/2}(\Gamma)\times H^{1/2}(\Gamma)$ to
$\cE_\Omega(\bg)\in\Hgdiv{\Omega}$. By assumption, the fact that $\cE_\Omega$ is a right-inverse
of $\tracegd{\Omega}$ and relation \eqref{trace_rel}, we find that
\begin{align*}
   0 &= \dualgd{\tracegd{}(\cE_\Omega(\bg))}{\bv}_\cS
     = \dualgd{\tracegd{\Omega}(\cE_\Omega(\bg))}{\bv}_\Gamma\\
     &= \duald{g_1}{\div\bv}_\Gamma - \dualg{g_2}{\bv}_\Gamma
     = \dual{g_1}{\div\bv}_\Gamma - \dual{g_2}{\bv\cdot\nn}_\Gamma
\end{align*}
where $\dual{\cdot}{\cdot}_\Gamma$ denotes the duality between
$H^{-1/2}(\Gamma)$ and $H^{1/2}(\Gamma)$ (in any order) with $L_2(\Gamma)$ as pivot space.
We conclude that $\div\bv=\bv\cdot\nn=0$ on $\Gamma$.
\end{proof}

\section{First-order variational formulation and DPG method} \label{sec_DPG1}

Denoting $\uone:=\bu$, we write \eqref{prob_org} as the first-order system
\begin{subequations} \label{prob_sys1}
\begin{align}
   \grad\ufour+\uone &= \ff, \label{sys1a}\\
   \div\uthree-\ufour  &=0,  \label{sys1b}\\
   \grad\utwo - \uthree &= 0, \label{sys1c}\\
   \div\uone - \utwo &=0 \label{sys1d}
\end{align}
\end{subequations}
in $\Omega$ subject to $\uone\cdot\nn=\utwo=0$ on $\Gamma$.
Testing \eqref{sys1a}--\eqref{sys1d} with piecewise smooth functions
$\vone$, $\vtwo$, $\vthree$, $\vfour$, respectively, and employing the trace operators
$\traceg{}$, $\traced{}$, we obtain
\begin{align*}
   \vdual{\uone}{\vone-\grad\vfour}_\cT
  -\vdual{\utwo}{\vfour+\div\vthree}_\cT
  -\vdual{\uthree}{\vthree+\grad\vtwo}_\cT
  -\vdual{\ufour}{\vtwo+\div\vone}_\cT
  &\quad\\
  +\duald{\traced{}(\uone)}{\vfour}_\cS
  +\dualg{\traceg{}(\utwo)}{\vthree}_\cS
  +\duald{\traced{}(\uthree)}{\vtwo}_\cS
  +\dualg{\traceg{}(\ufour)}{\vone}_\cS
  &=
   \vdual{\ff}{\vone}.
\end{align*}
We replace the traces by independent trace variables
\[
   \tuone = \traced{}(\uone),\
   \tutwo = \traceg{}(\utwo),\
   \tuthree = \traced{}(\uthree),\
   \tufour = \traceg{}(\ufour),
\]
and define the spaces
\begin{align*}
   \UU &:= \UU_0 \times \tUU\quad\text{with}\quad
   \UU_0:= \bL_2(\Omega)\times L_2(\Omega)\times \bL_2(\Omega)\times L_2(\Omega),\\
   &\hspace{8em}
   \tUU := H^{-1/2}_{0}(\cS)\times H^{1/2}_{0}(\cS)\times H^{-1/2}(\cS)\times H^{1/2}(\cS),
   \quad\text{and}\\
   \VV(\cT) &:= \Hdiv{\cT}\times H^1(\cT)\times \Hdiv{\cT}\times H^1(\cT)
\end{align*}
with (squared) norms
\begin{align*} 
   &\|(\uone,\utwo,\uthree,\ufour,\tuone,\tutwo,\tuthree,\tufour)\|_{\UU}^2
   := \|(\uone,\utwo,\uthree,\ufour)\|^2 + \|(\tuone,\tutwo,\tuthree,\tufour)\|_{\tUU}^2
   \\
   &:=
   \|\uone\|^2 + \|\utwo\|^2 + \|\uthree\|^2 + \|\ufour\|^2
   + \|\tuone\|_{-1/2,\cS}^2 + \|\tutwo\|_{1/2,\cS}^2 +
   \|\tuthree\|_{-1/2,\cS}^2 + \|\tufour\|_{1/2,\cS}^2
\end{align*}
and
\begin{align*} 
   \|(\vone,\vtwo,\vthree,\vfour)\|_{\VV(\cT)}^2
   &:=
   \|\vone\|_{\div\!,\cT}^2 + \|\vtwo\|_{1,\cT}^2 + \|\vthree\|_{\div\!,\cT}^2 + \|\vfour\|_{1,\cT}^2,
\end{align*}
respectively.
Then, our first variational formulation of \eqref{prob_org} reads:\\
\emph{Find $\uu=(\uone,\utwo,\uthree,\ufour,\tuone,\tutwo,\tuthree,\tufour)\in\UU$ such that}
\begin{align} \label{VF1}
   b(\uu,\vv) = L(\vv)\quad\forall\vv=(\vone,\vtwo,\vthree,\vfour)\in \VV(\cT),
\end{align}
in operator form
\[
   \uu\in\UU:\quad B\uu = L\quad\text{in}\ \VV(\cT)'.
\]
Here,
\begin{align*}
   b(\uu,\vv)
   :=&
   \vdual{\uone}{\vone-\grad\vfour}_\cT
  -\vdual{\utwo}{\vfour+\div\vthree}_\cT
  -\vdual{\uthree}{\vthree+\grad\vtwo}_\cT
  -\vdual{\ufour}{\vtwo+\div\vone}_\cT
  \\
  &+\duald{\tuone}{\vfour}_\cS
  +\dualg{\tutwo}{\vthree}_\cS
  +\duald{\tuthree}{\vtwo}_\cS
  +\dualg{\tufour}{\vone}_\cS
\end{align*}
and
\[
   L(\vv) := \vdual{\ff}{\vone}.
\]

\begin{theorem} \label{thm_stab1}
For given $\ff\in \bL_2(\Omega)$ there exists
a unique solution $\uu=(\uone,\utwo,\uthree,\ufour,\tuone,\tutwo,\tuthree,\tufour)\in\UU$
to \eqref{VF1}. It is uniformly bounded,
\[
   \|\uu\|_{\UU} \lesssim \|\ff\|
\]
with a hidden constant that is independent of $\ff$ and $\cT$.
Furthermore, $\bu:=\uone$ solves \eqref{prob_org}.
\end{theorem}

For a proof of this theorem we refer to \S\ref{sec_pf1}.

In order to define our approximation scheme, we consider a (family of) discrete subspace(s)
$\UU_h\subset\UU$ (being set on a sequence of meshes $\cT$)
and introduce the \emph{trial-to-test operator} $\ttt:\;\UU\to \VV(\cT)$ by
\begin{align*}
   \ip{\ttt(\uu)}{\vv}_{\VV(\cT)} = b(\uu,\vv)\quad\forall\vv\in \VV(\cT).
\end{align*}
Here, $\ip{\cdot}{\cdot}_{\VV(\cT)}$ denotes the inner product of $\VV(\cT)$ that induces the
norm defined previously.

Then, our first DPG scheme for problem~\eqref{prob_org} is:
\emph{Find $\uu_h\in \UU_h$ such that}
\begin{align} \label{DPG1}
   b(\uu_h,\ttt\deltauu) = L(\ttt\deltauu) \quad\forall\deltauu\in \UU_h.
\end{align}

The following theorem states its quasi-optimal approximation property.

\begin{theorem} \label{thm_DPG1}
Let $\ff\in \bL_2(\Omega)$ be given.
For any finite-dimensional subspace $\UU_h\subset \UU$
there exists a unique solution $\uu_h\in \UU_h$ to \eqref{DPG1}. It satisfies the quasi-optimal
error estimate
\[
   \|\uu-\uu_h\|_{\UU} \lesssim \|\uu-\ww\|_{\UU}
   \quad\forall\ww\in \UU_h
\]
with a hidden constant that is independent of $\cT$.
\end{theorem}

A proof of this theorem is given in \S\ref{sec_pf1}.

\subsection{Fully discrete method} \label{sec_discrete1}

In practice, the action of the trial-to-test operator $\ttt$ has to be approximated.
To this end we replace the test space by a sufficiently rich finite-dimensional subspace
$\VV_h(\cT)\subset\VV(\cT)$,
and use the approximated operator $\ttt_h:\;\UU_h\to\VV_h(\cT)$ defined by
\begin{align*}
   \ip{\ttt_h(\uu)}{\vv}_{\VV(\cT)} = b(\uu,\vv)\quad\forall\vv\in \VV_h(\cT).
\end{align*}
The fully discrete DPG method then is
\begin{align} \label{DPG1_discrete}
   \uu_h\in\UU_h:\quad b(\uu_h,\ttt_h\deltauu) = L(\ttt_h\deltauu) \quad\forall\deltauu\in \UU_h.
\end{align}
Discrete stability and quasi-optimal convergence of this scheme is guaranteed by the
existence of a (sequence of) Fortin operator(s) $\Pi_F:\;\VV(\cT)\to\VV_h(\cT)$, i.e.,
$\Pi_F$ is uniformly bounded with respect to the underlying (sequence of) space(s)
$\VV_h(\cT)$ and satisfies
\[
   b(\uu_h,\vv-\Pi_F\vv) = 0\quad\forall \uu_h\in\UU_h,\ \vv\in\VV(\cT),
\]
see \cite{GopalakrishnanQ_14_APD} for details.

In order to be specific, let us fix the discretization.
We consider a polygonal domain $\Omega\subset\R^d$ ($d\ge 2$) and meshes $\cT$ of shape-regular
simplicial elements:
\[
   \sup_{T\in\cT} \frac {\diam(T)^d}{|T|}\quad\text{is uniformly bounded.}
\]
In the following, $\cF_T$ denotes the set of faces of an element $T\in\cT$.
For a non-negative integer $p$ let us define the discrete spaces
\begin{align*}
   \cP^p(T) &:= \{v:\; T\to\R;\; v\ \text{is a polynomial of degree}\ p\}\quad (T\in\cT),\\
   \cP^p(\cT) &:= \{v\in L_2(\Omega);\; v|_T\in\cP^p(T)\ \forall T\in\cT\},\\
   \cP^p(F) &:= \{v:\;F\to\R;\; v\ \text{is a polynomial of degree}\ p\}\quad (F\in\cF_T, T\in\cT),\\
   \cP^p(\cF_T) &:= \{v\in L_2(\partial T);\; v|_F\in\cP^p(F)\ \forall F\in\cF_T\}\quad (T\in\cT),
   \\[1em]
   U^{\mathrm{div},p}(\cT)
      &:= \{\bv\in\Hdiv{\Omega};\; (\bv\cdot\nn)|_{\partial T}\in \cP^p(\cF_T)\ \forall T\in\cT\},\\
   \cP^p(\cS) &:= \traced{}( U^{\mathrm{div},p}(\cT)),\quad
   \cP^p_0(\cS) := \traced{}( U^{\mathrm{div},p}(\cT)\cap\Hdivz{\Omega}),
   \\[1em]
   U^{\mathrm{grad},p}(\cT)
      &:= \{v\in H^1(\Omega);\; v|_{\partial T}\in\cP^p(\cF_T)\ \forall T\in\cT\},\\
   \cP^{p,c}(\cS) &:= \traceg{}(U^{\mathrm{grad},p}(\cT)),\quad
   \cP^{p,c}_0(\cS) := \traceg{}(U^{\mathrm{grad},p}(\cT)\cap H^1_0(\Omega)).
\end{align*}
Then we consider the approximation space
\begin{align} \label{Uh1}
   \UU_h:=\UU_h^p:= \cP^p(\cT)^d\times\cP^p(\cT)\times\cP^p(\cT)^d\times\cP^p(\cT)
   \times \cP^p_0(\cS)\times \cP^{p+1,c}_0(\cS) \times\cP^p(\cS)\times \cP^{p+1,c}(\cS)
\end{align}
and discrete test space
\begin{align} \label{Vh1}
   \VV_h(\cT):=\VV_h^p(\cT):=
   \cP^{p+2}(\cT)^d\times \cP^{p+d+1}(\cT)\times \cP^{p+2}(\cT)^d\times \cP^{p+d+1}(\cT).
\end{align}
We employ the operators $\Pi^\mathrm{div}_{p+2}$ and $\Pi^\mathrm{grad}_{p+d+1}$
from \cite[Section~3.4]{GopalakrishnanQ_14_APD}, and using the results reported there it
is clear that
\[
   \Pi_F(\vv):=(\Pi^\mathrm{div}_{p+2}\vone,\Pi^\mathrm{grad}_{p+d+1}\vtwo,
                \Pi^\mathrm{div}_{p+2}\vthree,\Pi^\mathrm{grad}_{p+d+1}\vfour)
\]
for $\vv=(\vone,\vtwo,\vthree,\vfour)\in\VV(\cT)$ is a Fortin operator as required for our case.
For details we refer to \cite[Lemmas~3.2,3.3]{GopalakrishnanQ_14_APD}. Note the
increased polynomial degree by $d+1$ instead of $d$ as in \cite{GopalakrishnanQ_14_APD}
for the approximation of $H^1(\cT)$. This is necessary due to the presence of the terms
$\vdual{\utwo}{\vfour}$ and $\vdual{\ufour}{\vtwo}$ in the bilinear form,
contrary to the Poisson case in \cite{GopalakrishnanQ_14_APD} without reaction term.


Combining the established existence of a Fortin operator with Theorem~\ref{thm_DPG1}, we conclude
the quasi-optimal convergence of the fully-discrete DPG scheme,
cf.~\cite[Theorem~2.1]{GopalakrishnanQ_14_APD}.

\begin{cor} \label{cor_DPG1_discrete}
Let $\ff\in \bL_2(\Omega)$ and $p\in\N_0$ be given. Selecting the approximation and test spaces
$\UU_h$, $\VV_h(\cT)$ as in \eqref{Uh1}, \eqref{Vh1}, respectively, system \eqref{DPG1_discrete}
has a unique solution $\uu_h\in \UU_h$. It satisfies the quasi-optimal
error estimate
\[
   \|\uu-\uu_h\|_{\UU} \lesssim \|\uu-\ww\|_{\UU}
   \quad\forall\ww\in \UU_h
\]
with a hidden constant that is independent of $\cT$.
\end{cor}

\subsection{Proofs for the first-order formulation} \label{sec_pf1}

In this section we prove Theorems~\ref{thm_stab1},~\ref{thm_DPG1}. This requires
some preliminary steps given next.

\subsubsection{Stability of the adjoint problem} \label{sec_adj1}

The adjoint problem to \eqref{prob_sys1} with continuous test functions
\[
   (\vone,\vtwo,\vthree,\vfour)\in \VV_0(\Omega)
   := \Hdivz{\Omega}\times H^1_0(\Omega)\times \Hdiv{\Omega}\times H^1(\Omega)
\]
reads as follows:
Given $\gone,\gthree\in \bL_2(\Omega)$ and $\gtwo,\gfour\in L_2(\Omega)$, 
\emph{find $(\vone,\vtwo,\vthree,\vfour)\in \VV_0(\Omega)$ such that}
\begin{align} \label{adj1}
   \vone-\grad\vfour = \gone,\quad
   \vfour+\div\vthree = \gtwo,\quad
   \vthree+\grad\vtwo = \gthree,\quad
   \vtwo+\div\vone  = \gfour.
\end{align}
It is a fully non-homogeneous version of \eqref{prob_sys1} with some sign changes.
We show that this problem is well posed.

\begin{lemma} \label{la_adj1_well}
There is a unique solution $\vv=(\vone,\vtwo,\vthree,\vfour)\in \VV_0(\Omega)$ to \eqref{adj1},
and the bound
\[
   \|\vv\|_{\VV(\cT)}^2 \lesssim
   \|\gone\|^2 + \|\gtwo\|^2 + \|\gthree\|^2 + \|\gfour\|^2
\]
holds with a constant that is independent of the given data.
\end{lemma}

\begin{proof}
Eliminating $\vone$ and $\vthree$, \eqref{adj1} becomes
\[
   \vfour+\div(\gthree-\grad\vtwo)=\gtwo,\quad \vtwo+\div(\gone+\grad\vfour)=\gfour,
\]
and in variational form: {\em Find $(\vtwo,\vfour)\in H^1_0(\Omega)\times H^1(\Omega)$ such that}
\begin{align} \label{adj1_mixed}
   a(\vtwo,\vfour;\dvtwo,\dvfour) = \vdual{\gtwo}{\dvtwo}+\vdual{\gthree}{\grad\dvtwo}
                                  + \vdual{\gfour}{\dvfour} + \vdual{\gone}{\grad\dvfour}
\end{align}
{\em for any} $(\dvtwo,\dvfour)\in H^1_0(\Omega)\times H^1(\Omega)$ where
\begin{align*}
   a(\vtwo,\vfour;\dvtwo,\dvfour) :=
   \vdual{\grad\vtwo}{\grad\dvtwo} + \vdual{\vfour}{\dvtwo} +
   \vdual{\vtwo}{\dvfour} - \vdual{\grad\vfour}{\grad\dvfour}.
\end{align*}
Note that the discarded duality between the traces on $\Gamma$ of $(\gone+\grad\vfour)\cdot\nn$
and $\dvfour$ implies the required homogeneous boundary condition for $\vone:=\gone+\grad\vfour$.
Problem \eqref{adj1_mixed} is a mixed system of saddle point form.
It is somewhat non-standard due to the presence of the negative semi-definite form
$-\vdual{\grad\vfour}{\grad\dvfour}$. We show its well-posedness.
Then, $\vv:=(\vone,\vtwo,\vthree,\vfour)$ with
$\vone:=\gone+\grad\vfour$ and $\vthree:=\gthree-\grad\vtwo$ satisfies $\vv\in\VV_0(\Omega)$,
solves \eqref{adj1}, and is bounded as claimed.

Noting the boundedness of the appearing linear and bilinear forms, to show well-posedness
of \eqref{adj1_mixed} it is enough to verify the inf-sup condition
\begin{align} \label{adj1_infsup}
   \sup_{0\not=(\dvtwo,\dvfour)\in H^1_0(\Omega)\times H^1(\Omega)}
   \frac {a(\vtwo,\vfour;\dvtwo,\dvfour)}{\|\grad\dvtwo\| + \|\dvfour\|_1}
   \gtrsim \|\grad\vtwo\| + \|\vfour\|_1
   \quad\forall (\vtwo,\vfour)\in H^1_0(\Omega)\times H^1(\Omega)
\end{align}
and injectivity
\begin{align*} 
   a(\dvtwo,\dvfour;\vtwo,\vfour)=0\quad\forall (\dvtwo,\dvfour)\in H^1_0(\Omega)\times H^1(\Omega)
   \quad\Rightarrow\quad \vtwo=\vfour=0
\end{align*}
for $(\vtwo,\vfour)\in H^1_0(\Omega)\times H^1(\Omega)$.
The injectivity is immediate since
$a(\vtwo,-\vfour;\vtwo,\vfour)=\|\grad\vtwo\|^2+\|\grad\vfour\|^2=0$ implies that
$\vtwo=0$, $\vfour=c\in\R$, and
$a(\dvtwo,0;0,c)=\vdual{\dvtwo}{c}=0$ $\forall\dvtwo\in H^1_0(\Omega)$ means that $c=0$.

We are left with proving \eqref{adj1_infsup}. To this end we define
\[
   \|(\vtwo,\vfour)\|_a^2 := \|\grad\vtwo\|^2 + \|\grad\vfour\|^2 + |a(\vtwo,\vfour;w,0)|^2,
   \quad (\vtwo,\vfour)\in H^1_0(\Omega)\times H^1(\Omega),
\]
where $w\in H^1_0(\Omega)$ is a fixed function solving $-\Delta w=1$ in $\Omega$.
One verifies that $\|\cdot\|_a$ is a norm in $H^1_0(\Omega)\times H^1(\Omega)$
and that the equivalence
\[
   \|(\vtwo,\vfour)\|_a + \|\vfour\|
   \simeq
   \|\grad\vtwo\| + \|\grad\vfour\| + \|\vfour\|
   \quad\forall (\vtwo,\vfour)\in H^1_0(\Omega)\times H^1(\Omega)
\]
holds true.
By Rellich's compactness of the embedding $H^1(\Omega)\to L_2(\Omega)$ and
the Peetre--Tartar lemma, see, e.g., \cite[Theorem I.2.1]{GiraultR_86_FEM},
we also have the equivalence
\[
   \|(\vtwo,\vfour)\|_a
   \simeq
   \|\grad\vtwo\| + \|\grad\vfour\| + \|\vfour\|
   \quad\forall (\vtwo,\vfour)\in H^1_0(\Omega)\times H^1(\Omega).
\]
The inf-sup property follows by selecting $\dvtwo:=\vtwo+a(\vtwo,\vfour;w,0)w$,
$\dvfour:=-\vfour$, calculating
\[
   a(\vtwo,\vfour;\dvtwo,\dvfour)
   = \|\grad\vtwo\|^2 + \|\grad\vfour\|^2 + |a(\vtwo,\vfour;w,0)|^2
   = \|(\vtwo,\vfour)\|_a^2
\]
and bounding
\[
   \|\grad\dvtwo\|^2 + \|\grad\dvfour\|^2 + \|\dvfour\|^2
   \lesssim
   \|\grad\vtwo\|^2 + \|\vfour\|_1^2 + |a(\vtwo,\vfour;w,0)|^2 \|\grad w\|^2
   \lesssim
   \|\grad\vtwo\|^2 + \|\vfour\|_1^2.
\]
The hidden constants only depend on $\Omega$.
\end{proof}

\begin{lemma} \label{la_inj1}
The adjoint operator $B^*:\;\VV(\cT)\to \UU'$ is injective.
\end{lemma}

\begin{proof}
The proof is standard. We have to show that, if $\vv\in\VV(\cT)$ satisfies
$b(\duu,\vv)=0$ $\forall\duu\in\UU$
then $\vv=0$. Selecting functions $\duu$ with only one of the trace components non-zero
and all the field variables zero, Proposition~\ref{prop_can_jumps} shows that
$\vv\in\VV_0(\Omega)$. Furthermore, $\vv$ solves \eqref{adj1} with homogeneous data so that
$\vv=0$ by Lemma~\ref{la_adj1_well}.
\end{proof}

\subsubsection{Proofs of Theorems~\ref{thm_stab1},~\ref{thm_DPG1}} \label{sec_pf1}

The well-posedness of \eqref{VF1} follows by standard arguments.
\begin{enumerate}
\item {\bf Boundedness of the functional and bilinear form.}
By definition of the norms in $\UU$ and $\VV(\cT)$, and using Proposition~\ref{prop_can_norms},
the uniform boundedness of $b(\cdot,\cdot)$ and $L$ is immediate.
\item {\bf Injectivity.} This is Lemma~\ref{la_inj1}.
\item {\bf Inf-sup condition.} By \cite[Theorem~3.3]{CarstensenDG_16_BSF}, the inf-sup condition
\begin{align} \label{infsup1}
   \sup_{0\not=\vv\in \VV(\cT)}
   \frac {b(\uu;\vv)}{\|\vv\|_{\VV(\cT)}}
   \gtrsim \|\uu\|_{\UU}
   \quad\forall\uu=(\uu_0,\tuu)\in \UU_0\times\tUU= \UU
\end{align}
follows from the inf-sup conditions
\begin{align}
   \label{infsup1b}
   \sup_{0\not=\vv\in \VV(\cT)}
   \frac{b(0,\tuu;\vv)}{\|\vv\|_{\VV(\cT)}}
   \gtrsim
   \|\tuu\|_{\tUU} \quad\forall\tuu\in\tUU
\end{align}
and
\begin{align}
   \label{infsup1a}
   \sup_{0\not=\vv\in \VV_0(\Omega)}
   &\frac{b(\uu_0,0;\vv)}{\|\vv\|_{\VV(\cT)}}
   \gtrsim
   \|\uu_0\| \quad\forall\uu_0\in\UU_0.
\end{align}
We note that Proposition~\ref{prop_can_jumps} implies that $\VV_0(\Omega)$ is the correct
test space in \eqref{infsup1a}. It corresponds to $Y_0$ in \cite{CarstensenDG_16_BSF}, see
relation (17) there.

Relation \eqref{infsup1b} is true (with constant $1$) by Proposition~\ref{prop_can_norms},
and inf-sup condition \eqref{infsup1a} holds due to Lemma~\ref{la_adj1_well}:
Selecting data
\(
   (\gone,\gtwo,\gthree,\gfour):= (\uone,-\utwo,-\uthree,-\ufour)
\)
in \eqref{adj1} with solution $\vv^*\in\VV_0(\Omega)$, we bound
\begin{align*}
   &\sup_{0\not=\vv\in \VV_0(\Omega)}
   \frac{b(\uu_0,0;\vv)}{\|\vv\|_{\VV(\cT)}}
   \ge
   \frac{\vdual{\uone}{\gone} - \vdual{\utwo}{\gtwo}
   - \vdual{\uthree}{\gthree} - \vdual{\ufour}{\gfour}}
        {\|\vv^*\|_{\VV(\cT)}}
   \gtrsim \|\uu_0\|.
\end{align*}
\end{enumerate}
We conclude that there is a unique and stable solution $\uu=(\uu_0,\tuu)\in \UU_0\times\tUU$
of \eqref{VF1}. It remains to note that $\uu_0$ solves \eqref{prob_sys1},
satisfies the homogeneous boundary conditions, and
$\tuu=(\traced{}(\uone),\traceg{}(\utwo),\traced{}(\uthree),\traceg{}(\ufour))$.
Furthermore, $\bu:=\uone$ solves \eqref{prob_org}.
This finishes the proof of Theorem~\ref{thm_stab1}.

Theorem~\ref{thm_DPG1} follows in the standard way. Noting that the DPG scheme minimizes
the residual
\(
   \|B(\uu-\uu_h)\|_{\VV(\cT)'},
\)
its quasi-optimal convergence in the norm $\|\cdot\|_{\UU}$ is due to the equivalence
of this norm and the residual norm $\|B\cdot\|_{\VV(\cT)'}$.
Specifically, the boundedness $\|B\cdot\|_{\VV(\cT)'}\lesssim \|\cdot\|_{\UU}$
is due to the boundedness of $b(\cdot,\cdot)$, and the inverse estimate is proved by
inf-sup property \eqref{infsup1}.

\section{Second-order variational formulation and DPG method} \label{sec_DPG2}

In this section, we use the same notation as in \S\ref{sec_DPG1} for spaces, norms,
bilinear form, functional $L$, operator $B$, and trial-to-test operator $\ttt$.

Denoting $\bw:=-\grad\div\bu$, we write \eqref{prob_org} as the second-order system
\begin{subequations} \label{prob_sys2}
\begin{align}
   -\grad\div\bw+\bu &= \ff, \label{sys2a}\\
   \grad\div\bu+\bw  &=0,  \label{sys2b}
\end{align}
\end{subequations}
in $\Omega$ subject to $\bu\cdot\nn=\div\bu=0$ on $\Gamma$.
Testing \eqref{sys2a} and \eqref{sys2b} with piecewise smooth functions
$\bv$ and $-\btau$, respectively, and employing the trace operator
$\tracegd{}$, we obtain
\begin{align*}
   \vdual{\bu}{\bv-\grad\div\btau}_\cT
  -\vdual{\bw}{\btau+\grad\div\bv}_\cT
  +\dualgd{\tracegd{}(\bu)}{\btau}_\cS
  +\dualgd{\tracegd{}(\bw)}{\bv}_\cS
  &=
   \vdual{\ff}{\bv}.
\end{align*}
We introduce independent trace variables
\[
   \tbu = \tracegd{}(\bu),\
   \tbw = \tracegd{}(\bw),\
\]
and define the spaces
\begin{align*}
   &\UU := \UU_0 \times \tUU
       := \bigl(\bL_2(\Omega)\times \bL_2(\Omega)\bigr)
   \times \bigl(\trHgdz{\cS}\times \trHgd{\cS}\bigr),\\
   &\VV(\cT) := \Hgdiv{\cT}\times \Hgdiv{\cT}
\end{align*}
with (squared) norms
\begin{align*} 
   &\|(\bu,\bw,\tbu,\tbw)\|_{\UU}^2
   := \|(\bu,\bw)\|^2 + \|(\tbu,\tbw)\|_{\tUU}^2
   := \|\bu\|^2 + \|\bw\|^2 + \|\tbu\|_{-1/2,1/2,\cS}^2 + \|\tbw\|_{-1/2,1/2,\cS}^2
\end{align*}
and
\begin{align*} 
   \|(\bv,\btau)\|_{\VV(\cT)}^2
   &:=
   \|\bv\|_{\grad\div\!,\cT}^2 + \|\btau\|_{\grad\div\!,\cT}^2,
\end{align*}
respectively.
Then, our second variational formulation of \eqref{prob_org} reads:
\emph{Find $\uu=(\bu,\bw,\tbu,\tbw)\in\UU$ such that}
\begin{align} \label{VF2}
   b(\uu,\vv) = L(\vv)\quad\forall\vv=(\bv,\btau)\in \VV(\cT),
\end{align}
in operator form
\[
   \uu\in\UU:\quad B\uu = L\quad\text{in}\ \VV(\cT)'.
\]
Here,
\begin{align*}
   b(\uu,\vv)
   :=&
   \vdual{\bu}{\bv-\grad\div\btau}_\cT
  -\vdual{\bw}{\btau+\grad\div\bv}_\cT
  +\dualgd{\tbu}{\btau}_\cS
  +\dualgd{\tbw}{\bv}_\cS
\end{align*}
and
\[
   L(\vv) := \vdual{\ff}{\bv}.
\]

\begin{theorem} \label{thm_stab2}
For given $\ff\in\bL_2(\Omega)$ there exists
a unique solution $\uu=(\bu,\bw,\tbu,\tbw)\in\UU$ to \eqref{VF2}. It is uniformly bounded,
\[
   \|\uu\|_{\UU} \lesssim \|\ff\|
\]
with a hidden constant that is independent of $\ff$ and $\cT$.
Furthermore, $\bu$ solves \eqref{prob_org}.
\end{theorem}

For a proof of this theorem we refer to \S\ref{sec_pf2}.

To define the discretization scheme, we consider as before a (family of) discrete subspace(s)
$\UU_h\subset\UU$ and use the trial-to-test operator $\ttt:\;\UU\to \VV(\cT)$ defined by
\begin{align*}
   \ip{\ttt(\uu)}{\vv}_{\VV(\cT)} = b(\uu,\vv)\quad\forall\vv\in \VV(\cT)
\end{align*}
where $\ip{\cdot}{\cdot}_{\VV(\cT)}$ denotes the inner product of the new space $\VV(\cT)$,
\begin{align*}
   \ip{(\bv_1,\bv_2)}{(\dbv_1,\dbv_2)}_{\VV(\cT)}
   := \sum_{j=1}^2 \vdual{\bv_j}{\dbv_j}+\vdual{\grad\div\bv_j}{\grad\div\dbv_j}_\cT
\end{align*}
for $(\bv_1,\bv_2), (\dbv_1,\dbv_2)\in \VV_h(\cT)$.
The second DPG scheme for problem~\eqref{prob_org} is:
\emph{Find $\uu_h\in \UU_h$ such that}
\begin{align} \label{DPG2}
   b(\uu_h,\ttt\deltauu) = L(\ttt\deltauu) \quad\forall\deltauu\in \UU_h.
\end{align}
Again, it delivers a quasi-best approximation.

\begin{theorem} \label{thm_DPG2}
Let $\ff\in \bL_2(\Omega)$ be given.
For any finite-dimensional subspace $\UU_h\subset \UU$
there exists a unique solution $\uu_h\in \UU_h$ to \eqref{DPG2}. It satisfies the quasi-optimal
error estimate
\[
   \|\uu-\uu_h\|_{\UU} \lesssim \|\uu-\ww\|_{\UU}
   \quad\forall\ww\in \UU_h
\]
with a hidden constant that is independent of $\cT$.
\end{theorem}

A proof of this theorem is given in \S\ref{sec_pf2}.

\subsection{Fully discrete method} \label{sec_discrete2}

In the case of the second-order formulation we restrict our fully discrete analysis
to two space dimensions ($d=2$) and lowest order approximations. Using the notation
for (piecewise) polynomial spaces from \S\ref{sec_discrete1}, our approximation space is
\begin{align} \label{Uh2}
   \UU_h:=\cP^0(\cT)^2\times\cP^0(\cT)^2
   \times \cP^0_0(\cS)\times \cP^{1,c}_0(\cS) \times\cP^0(\cS)\times \cP^{1,c}(\cS)
\end{align}
and as enriched test space we choose
\begin{align} \label{Vh2}
   \VV_h(\cT):=
    \cP^{3}(\cT)^2\times \cP^{3}(\cT)^2.
\end{align}
Here, $\cT$ denotes a (family of) shape-regular triangular mesh(es). The notation
for our approximation scheme is as before:
\begin{align} \label{DPG2_discrete}
   \uu_h\in\UU_h:\quad b(\uu_h,\ttt_h\deltauu) = L(\ttt_h\deltauu) \quad\forall\deltauu\in \UU_h
\end{align}
where the approximated trial-to-test operator $\ttt_h:\;\UU_h\to\VV_h(\cT)$ is defined accordingly.

The existence of a Fortin operator in this case is due to Lemma~\ref{la_Fortin_local} (given
in \S\ref{sec_Fortin2} below), showing the existence of a local variant
$\PFgd{T}:\;\Hgdiv{T}\to \cP^3(T)^2$.
We simply define $\PFgd{}:\;\Hgdiv{\cT}\to \cP^3(\cT)^2$ as
$(\PFgd{}\bv)|_T:=\PFgd{T}\bv|_T$ ($T\in\cT$), and
\[
   \Pi_F(\vv):=(\PFgd{}\bv,\PFgd{}\btau)\quad\text{for}\quad \vv=(\bv,\btau)\in \VV(\cT).
\]
The statements of Lemma~\ref{la_Fortin_local} imply that this operator
has the properties of a Fortin operator.
Then, analogously as Corollary~\ref{cor_DPG1_discrete}, we conclude the quasi-optimal
convergence of the fully discrete DPG scheme.

\begin{cor} \label{cor_DPG2_discrete}
Let $\ff\in \bL_2(\Omega)$ be given. Selecting the approximation and test spaces
$\UU_h$, $\VV_h(\cT)$ as in \eqref{Uh2}, \eqref{Vh2}, respectively, system \eqref{DPG2_discrete}
has a unique solution $\uu_h\in \UU_h$. It satisfies the quasi-optimal
error estimate
\[
   \|\uu-\uu_h\|_{\UU} \lesssim \|\uu-\ww\|_{\UU}
   \quad\forall\ww\in \UU_h
\]
with a hidden constant that is independent of $\cT$.
\end{cor}

\subsection{Proofs for the second-order formulation} \label{sec_pf2}

The steps in this section are analogous to the ones in \S\ref{sec_pf1}.

\subsubsection{Stability of the adjoint problem} \label{sec_adj2}

The adjoint problem to \eqref{prob_sys2} with continuous test functions
\[
   (\bv,\btau)\in \VV_0(\Omega) := \Hgdivz{\Omega}\times \Hgdiv{\Omega}
\]
reads as follows:
Given $\bg_1,\bg_2\in \bL_2(\Omega)$ 
\emph{find $(\bv,\btau)\in \VV_0(\Omega)$ such that}
\begin{align} \label{adj2}
   \bv-\grad\div\btau = \bg_1,\quad
   \btau+\grad\div\bv = \bg_2.
\end{align}
It is a well-posed, fully non-homogeneous version of \eqref{prob_sys2}.

\begin{lemma} \label{la_adj2_well}
There is a unique solution $\vv=(\bv,\btau)\in \VV_0(\Omega)$ to \eqref{adj2},
and the bound
\[
   \|\vv\|_{\VV(\cT)}^2 \lesssim \|\bg_1\|^2 + \|\bg_2\|^2
\]
holds with a constant that is independent of the given data.
\end{lemma}

\begin{proof}
Eliminating $\bv$, \eqref{adj2} becomes
\[
   \btau+\grad\div(\bg_1+\grad\div\btau) = \bg_2,
\]
in variational form: {\em Find $\btau\in\Hgdiv{\Omega}$ such that}
\begin{align} \label{adj2_mixed}
   \vdual{\grad\div\btau}{\grad\div\dbtau} + \vdual{\btau}{\dbtau}
   = \vdual{\bg_2}{\dbtau} - \vdual{\bg_1}{\grad\div\dbtau}
   \quad\forall \dbtau\in \Hgdiv{\Omega}.
\end{align}
As for the first-order system (see the proof of Lemma~\ref{la_adj1_well}),
the discarded duality between the trace
$\tracegd{}(\bg_1+\grad\div\btau)$ and $\dbtau$ implies the homogeneous boundary condition
for $\bv:=\gone+\grad\div\btau$.
By the Lax--Milgram lemma, problem \eqref{adj2_mixed} is well posed.
Then $(\bv,\btau)$ with $\bv:=\gone+\grad\div\btau$ solves \eqref{adj2} and satisfies
the claimed bound.
\end{proof}

\begin{lemma} \label{la_inj2}
The adjoint operator $B^*:\;\VV(\cT)\to \UU'$ is injective.
\end{lemma}

\begin{proof}
The proof is identical to the one of Lemma~\ref{la_inj1},
using Proposition~\ref{prop_gd_jumps} and  Lemma~\ref{la_adj2_well} instead of
Proposition~\ref{prop_can_jumps} and  Lemma~\ref{la_adj1_well}, respectively.
\end{proof}

\subsubsection{Proofs of Theorems~\ref{thm_stab2},~\ref{thm_DPG2}} \label{sec_pf2}

To show the well-posedness of \eqref{VF2} we repeat the arguments from \S\ref{sec_pf1}.
The boundedness of the functional and bilinear form is identical to the previous case.
The injectivity of $B^*$ holds by Lemma~\ref{la_inj2}. In order to prove
the inf-sup condition it is enough to show
\begin{align}
   \label{infsup2b}
   \sup_{0\not=\vv\in \VV(\cT)}
   \frac{b(0,\tuu;\vv)}{\|\vv\|_{\VV(\cT)}}
   \gtrsim
   \|\tuu\|_{\tUU} \quad\forall\tuu\in\tUU
\end{align}
and
\begin{align}
   \label{infsup2a}
   \sup_{0\not=\vv\in \VV_0(\Omega)}
   &\frac{b(\uu_0,0;\vv)}{\|\vv\|_{\VV(\cT)}}
   \gtrsim
   \|\uu_0\| \quad\forall\uu_0\in\UU_0.
\end{align}
The test space in \eqref{infsup2a} is the appropriate one due to Proposition~\ref{prop_gd_jumps}.

Relation \eqref{infsup2b} is true (with constant $1$) by Proposition~\ref{prop_gd_norm},
and inf-sup condition \eqref{infsup2a} can be seen by using Lemma~\ref{la_adj2_well}
with data
\(
   (\bg_1,\bg_2):= (\bu,-\bw)
\)
in \eqref{adj2}.
We conclude that there is a unique and stable solution $\uu=(\uu_0,\tuu)\in \UU_0\times\tUU$
of \eqref{VF2}. It remains to note that $(\bu,\bw)$ solves \eqref{prob_sys2},
$\bu$ satisfies the homogeneous boundary conditions, and
$\tbu=\tracegd{}(\bu)$, $\tbw=\tracegd{}(\bw)$.
Furthermore, $\bu$ solves \eqref{prob_org}.
This finishes the proof of Theorem~\ref{thm_stab2}.

The proof of Theorem~\ref{thm_DPG2} is identical to the one of Theorem~\ref{thm_DPG1}.

\subsubsection{Fortin operator} \label{sec_Fortin2}

Let $\Tref$ denote the reference triangle with vertices $(0,0)$, $(1,0)$, $(0,1)$.
We start with a preliminary Fortin operator $\pPFgd{\Tref}$ on the reference element.
The corresponding operator on an element $T\in\cT$ is not uniformly bounded
with respect to $\diam(T)$ so that it has to be adapted. This will be studied in
Lemma~\ref{la_Fortin_local}.

\begin{lemma} \label{la_Fortin_ref}
There is a linear bounded operator $\pPFgd{\Tref}:\;\Hgdiv{\Tref}\to \cP^3(\Tref)^2$
that satisfies
for any $\bw\in\cP^0(\Tref)^2$ and $\tbw\in\cP^0(\partial\Tref)\times \cP^{1,c}(\partial\Tref)$
\begin{align*}
   \vdual{\bw}{(1-\pPFgd{\Tref})\bv}_\Tref = 
   \dualgd{\tbw}{(1-\pPFgd{\Tref})\bv}_{\partial\Tref} = 0\quad
   \quad\forall \bv\in\Hgdiv{\Tref}.
\end{align*}
\end{lemma}

\begin{proof}
For the construction of this operator we employ the strategy from Nagaraj \emph{et al.}
\cite{NagarajPD_17_CDF}, see also \cite{FuehrerH_19_FDD}.
For given $\bv\in \Hgdiv{\Tref}$ we define $\pPFgd{\Tref}\bv:=\bv^*$ where
\[
   (\bv^*,\bw,\tbv)\in
   \cP^3(\Tref)^2\times \cP^0(\Tref)^2\times
   \bigl(\cP^0(\partial\Tref)\times \cP^{1,c}(\partial\Tref)\bigr)
\]
is the solution to the mixed problem
\begin{alignat*}{3}
   &\ip{\bv^*}{\dbv}_{\grad\div\!,\Tref}
   + \vdual{\bw}{\dbv}_\Tref + \dualgd{\tbv}{\dbv}_{\partial\Tref} &&= 0
   &&\quad\forall\dbv\in\cP^3(\Tref)^2,\\
   &\vdual{\dbu}{\bv^*}_\Tref &&= \vdual{\dbu}{\bv}_\Tref
   &&\quad\forall\dbu\in\cP^0(\Tref)^2,\\
   &\dualgd{\dtbu}{\bv^*}_{\partial\Tref} &&= \dualgd{\dtbu}{\bv}_{\partial\Tref}
   &&\quad\forall\dtbu\in \cP^0(\partial\Tref)\times \cP^{1,c}(\partial\Tref).
\end{alignat*}
Using a canonical basis, one can show that the matrix of this problem is
invertible, thus giving the result.
\end{proof}

\begin{lemma} \label{la_Fortin_local}
For $T\in\cT$ there is a linear operator $\PFgd{T}:\;\Hgdiv{T}\to \cP^3(T)^2$
that satisfies
for any $\bw\in\cP^0(T)^2$ and $\tbw\in\cP^0(\partial T)\times \cP^{1,c}(\partial T)$
\begin{align*}
   \vdual{\bw}{(1-\PFgd{T})\bv}_T = 
   \vdual{\bw}{\grad\div(1-\PFgd{T})\bv}_T =
   \dualgd{\tbw}{(1-\PFgd{T})\bv}_{\partial T} = 0
   \quad\forall \bv\in\Hgdiv{T},
\end{align*}
and which is uniformly bounded with respect to $\diam(T)$,
\[
   \|\PFgd{T}\bv\|_{\grad\div\!,T} \lesssim \|\bv\|_{\grad\div\!,T}
   \quad\forall\bv\in\Hgdiv{T},\ T\in\cT.
\]
\end{lemma}

\begin{proof}
In this proof, $\Pi^p$ denotes the $L_2$-projection onto polynomials of degree $p$ (on the domain
under consideration).
Let $T\in\cT$ and $\bv\in\Hgdiv{T}$ be given.
Furthermore, let 
$\Tref\ni\widehat{\boldsymbol{x}}\mapsto B_T\widehat{\boldsymbol{x}}+\boldsymbol{a}_T\in T$
be an affine mapping with $\boldsymbol{a}_T\in\R^2$, $B_T\in\R^{2\times 2}$, and $J_T:=\det(B_T)$.

We will define $\PFgd{T}\bv$ by splitting $\bv$ into three components and treat them
differently. To this end we need three operators: $\cI$ in step (i) below,
$\Pi^\mathrm{div}_3$ in step (ii), already used in \S\ref{sec_discrete1}, and
$\pPFgd{T}$ in step (iii), based on $\pPFgd{\Tref}$.

\begin{enumerate}
\item[(i)]
Following Ern \emph{et al.} \cite[Theorem~3.2]{ErnGSV_ELG}, there is an interpolation operator
$\cI:\;\Hdiv{T}\to\mathcal{RT}^0(T):=
\{\bz:\;T\to\R^2;\;\exists a,b,c\in\R:\;\bz=(a,b)^T+c\boldsymbol{x}\ \forall\boldsymbol{x}\in T\}$
with
$\div\circ\cI=\Pi^0\circ\div$ on $\Hdiv{T}$ and boundedness
\begin{align} \label{I}
   \|\cI\bz\|_T\lesssim \|\bz\|_T + \diam(T)\|(1-\Pi^0)\div\bz\|_T
   \quad\forall\bz\in \Hdiv{T}
\end{align}
uniformly with respect to $\diam(T)$. Note that $\|\grad\div\cT\bz\|_T=0$ $\forall\bz\in\Hgdiv{T}$.

Then we write $\bv=\cI\bv + \bw$ with $\bw:=(1-\cI)\bv$ and use the Helmholtz decomposition
$\bw=\grad\phi+\bcurl\bpsi$ where $\phi\in H^1_0(T)$ solves $\Delta\phi=\div\bw$,
and $\bpsi\in\Hcurl{T}$ (with obvious definition of the space).
Integration by parts, the commuting property of $\cI$, the Bramble--Hilbert lemma and
Poincare's inequality prove that
\begin{align} \label{phi}
   \|\grad\phi\|_T^2 &= -\vdual{\div\bw}{\phi}_T = -\vdual{(1-\Pi^0)\div\bv}{\phi}_T\nonumber\\
   &\lesssim \|(1-\Pi^0)\div\bv\|_T \|\phi\|_T
   \lesssim \diam(T)^2\|\grad\div\bv\|_T \|\grad\phi\|_T,\nonumber\\
   \text{i.e.,}\qquad
   \|\grad\phi\|_T &\lesssim \diam(T)^2\|\grad\div\bv\|_T.
\end{align}

\item[(ii)]
We continue to use the Fortin operator
$\Pi^\mathrm{div}_{p}:\;\Hdiv{T}\to\cP^{p}(T)^2$ from \cite[Lemma~3.3]{GopalakrishnanQ_14_APD}
already employed in \S\ref{sec_discrete1}, in this case for $p=3$. It is uniformly bounded
in the $\Hdiv{T}$-norm and has the properties
\[
   \vdual{\dbu}{\Pi^\mathrm{div}_3\bz}_T = \vdual{\dbu}{\bz}_T,\quad
   \dualg{\dtu}{\Pi^\mathrm{div}_3\bz}_{\partial T} = \dualg{\dtu}{\bz}_{\partial T}
   \quad\forall\bz\in \Hdiv{T}
\]
for any $\dbu\in\cP^1(T)^2$ and $\dtu\in\cP^{2,c}(\partial T)$.
Furthermore, $\div\circ\Pi^\mathrm{div}_3=\Pi^2\circ\div$,
see~\cite[Proof of Lemma~3.3]{GopalakrishnanQ_14_APD}. On the one hand, this property implies that
\begin{align} \label{3}
   \dualgd{\dtbu}{\Pi^\mathrm{div}_3\bcurl\bpsi}_{\partial T} = \dualgd{\dtbu}{\bcurl\bpsi}_{\partial T}
   \quad\forall\dtbu\in \cP^0(\partial T)\times\cP^{1,c}(\partial T).
\end{align}
On the other hand, also using \eqref{I} and the Bramble--Hilbert lemma, we conclude that
\begin{align*}
   &\|\Pi^\mathrm{div}_3\bcurl\bpsi\|_T \lesssim \|\bcurl\bpsi\|_T + \|\div\bcurl\bpsi\|_T
   \le \|\bw\|_T \lesssim  \|\bv\|_T + \diam(T)^2\|\grad\div\bv\|_T,\\
   \text{and}\quad
   &\|\grad\div\Pi^\mathrm{div}_3\bcurl\bpsi\|_T = \|\grad\Pi^2\div\bcurl\bpsi\|_T = 0.
\end{align*}

\item[(iii)]
We transform the operator $\pPFgd{\Tref}$ from Lemma~\ref{la_Fortin_ref} through the
Piola map $\Tref\to T$. Denoting by $\wat{\bz}\in\Hgdiv{\Tref}$
the preimage of $\bz\in\Hgdiv{T}$ under the Piola transform, we define
\[
   \pPFgd{T}\bz:=|J_T|^{-1}B_T \pPFgd{\Tref}\wat{\bz}\circ F^{-1}_T.
\]
By scaling and Lemma~\ref{la_Fortin_ref} we bound
\begin{align*}
   &\|\pPFgd{T}\bz\|_T \simeq \|\pPFgd{\Tref}\wat{\bz}\|_\Tref
   \lesssim \|\bz\|_T + \diam(T)^2\|\grad\div\bz\|_T,\\
   &\|\grad\div\pPFgd{T}\bz\|_T \simeq \diam(T)^{-2}\|\grad\div\pPFgd{\Tref}\wat{\bz}\|_\Tref
   \lesssim \diam(T)^{-2}\|\bz\|_T + \|\grad\div\bz\|_T.
\end{align*}
These estimates together with \eqref{phi}, and noting that
$\grad\div\bw=\grad\div(1-\cI)\bv=\grad\div\bv$, yield
\begin{align*}
   &\|\pPFgd{T}\grad\phi\|_T \lesssim \|\grad\phi\|_T + \diam(T)^2 \|\grad\div\bw\|_T
                           \lesssim \diam(T)^2\|\grad\div\bv\|_T,\\
   &\|\grad\div\pPFgd{T}\grad\phi\|_T \lesssim \diam(T)^{-2}\|\grad\phi\|_T + \|\grad\div\bw\|_T
                                     \lesssim \|\grad\div\bv\|_T.
\end{align*}
\end{enumerate}

{\bf Definition of $\PFgd{T}\bv$.}
Finally, we define $\PFgd{T}\bv:=\cI\bv+\Pi^\mathrm{div}_3\bcurl\bpsi+\pPFgd{T}\grad\phi$.
By steps (i)--(iii) it is clear that this operator is uniformly bounded. It remains to verify
the orthogonality relations. By the orthogonality properties of $\Pi^\mathrm{div}_3$ and
$\pPFgd{T}$ we immediately verify that
\begin{align*}
   \vdual{\dbu}{\PFgd{T}\bv}_T=\vdual{\dbu}{\cI\bv+\Pi^\mathrm{div}_3\bcurl\bpsi+\pPFgd{T}\grad\phi}_T
   =\vdual{\dbu}{\cI\bv+\bcurl\bpsi+\grad\phi}_T=\vdual{\dbu}{\bv}_T
\end{align*}
for any $\dbu\in\cP^0(T)^2$ and, recalling Lemma~\ref{la_Fortin_ref} and \eqref{3},
\begin{align*}
   \dualgd{\dtbu}{\PFgd{T}\bv}_{\partial T}
   &=
   \dualgd{\dtbu}{\cI\bv+\Pi^\mathrm{div}_3\bcurl\bpsi+\pPFgd{T}\grad\phi}_{\partial T}\\
   &=
   \dualgd{\dtbu}{\cI\bv+\bcurl\bpsi+\grad\phi}_{\partial T}
   =
   \dualgd{\dtbu}{\bv}_{\partial T}
\end{align*}
for any $\dtbu\in\cP^0(\partial T)\times\cP^{1,c}(\partial T)$. The remaining
orthogonality relation is checked by using the one we have just seen, and trace operator $\tracegd{T}$:
\begin{align*}
   \vdual{\dbu}{\grad\div\PFgd{T}\bv}_T
   &= 
   \vdual{\grad\div\dbu}{\PFgd{T}\bv}_T + \dualgd{\tracegd{T}\dbu}{\PFgd{T}\bv}_{\partial T}\\
   &=
   \dualgd{\tracegd{T}\dbu}{\bv}_{\partial T}
   =
   \vdual{\dbu}{\grad\div\bv}_T \quad\forall\dbu\in\cP^0(T)^2
\end{align*}
since $\tracegd{T}(\cP^0(T)^2)\subset \cP^0(\partial T)\times\cP^{1,c}(\partial T)$.
\end{proof}

\section{Numerical results} \label{sec_num}

We consider both DPG schemes for two examples in two dimensions,
one with smooth solution and one with a corner singularity. For both schemes we
use lowest order spaces, i.e., discrete spaces $\UU_h$ and $\VV_h(\cT)$ as in
\S\ref{sec_discrete1} with $p=0$ for the scheme based on the first-order system,
and discrete spaces as in \S\ref{sec_discrete2} for the method based on the second-order system.

\bigskip
{\bf Example with smooth solution.} In this case, we select $\Omega=(0,1)^2$, the
manufactured solution
$\bu(\boldsymbol{x}):=(x_1^2(x_1-1)^2x_2^2(x_2-1)^2,\sin^2(\pi x_1)\sin^2(\pi x_2))^T$
for $\boldsymbol{x}=(x_1,x_2)\in\Omega$, and calculate $\ff$ accordingly.

Considering a sequence of quasi-uniform triangular meshes, we expect both schemes
to deliver approximation errors of the order $O(h)=O(\dim(\UU_h)^{-1/2})$. This is
confirmed in Figure~\ref{fig_smooth}, with error curves for the first scheme on the left,
and the second on the right. Here and below, $\eta$ stands for the corresponding
residual $\|B\uu_h-L\|_{\VV_h(\cT)'}$, and note that $\uone=\bu$, $\bw=-\uthree=-\grad\div\bu$.

\begin{figure}
\begin{center}
\includegraphics[height=0.405\textwidth]{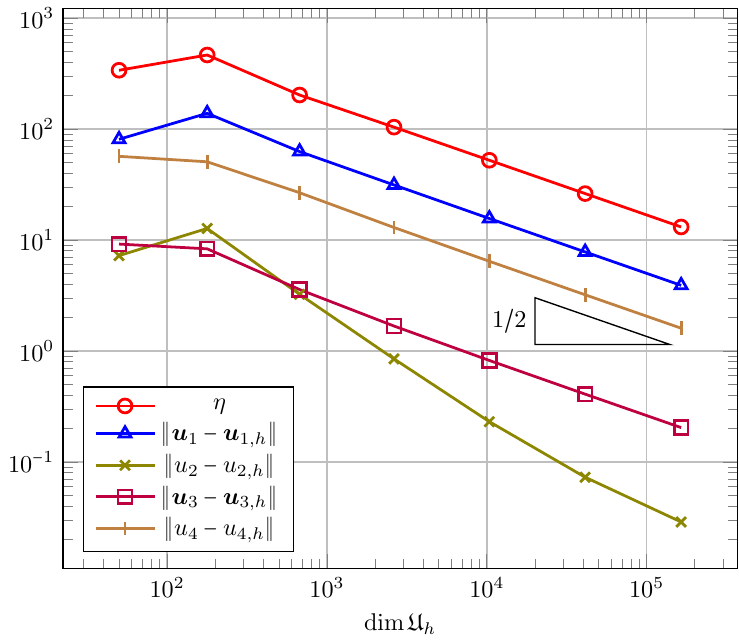}\quad
\includegraphics[height=0.4\textwidth]{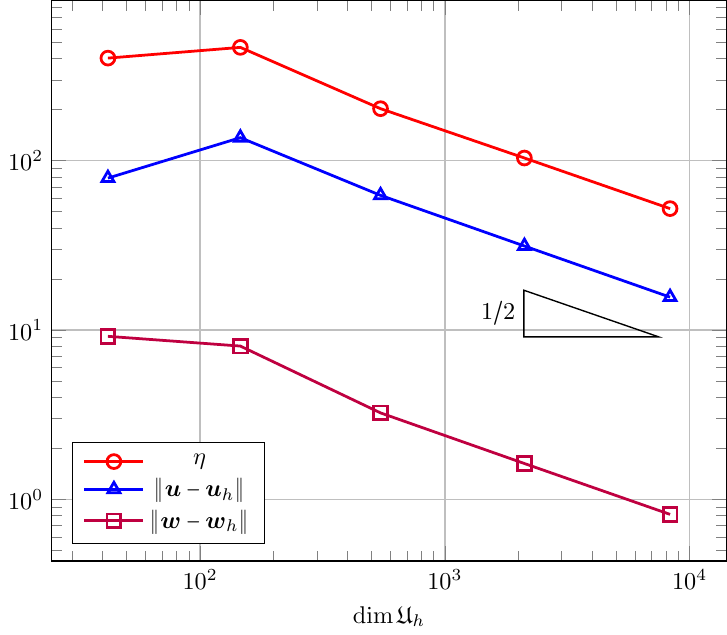}
\end{center}
\caption{Smooth example, quasi-uniform meshes. Left: first-order system, right: second-order system}
\label{fig_smooth}
\end{figure}

\bigskip
{\bf Example with singular solution.} We select the L-shaped domain
\[
   \Omega=\bigl\{(x_1,x_2);\; |x_1|+|x_2|< \frac{\sqrt{2}}4 \bigr\}\setminus
          \bigl\{(x_1,x_2);\; |x_1+\frac{\sqrt{2}}4|+|x_2|\le \frac{\sqrt{2}}4\bigr\}
\]
with initial mesh as on the left of Figure~\ref{fig_mesh}, solution $\bu:=\bcurl v$ with
$v(r,\varphi):=r^{2/3}\cos(2\varphi/3)$ ($(r,\varphi)$ denoting polar coordinates centered
at the origin so that $v=0$ on the boundary edges meeting at the origin),
and corresponding right-hand side function $\ff$.
This solution satisfies the boundary conditions $\div\bu=0$ on $\Gamma$ and
$\bu\cdot\nn=0$ along the edges meeting at the incoming corner. The non-homogeneous
trace of $\bu\cdot\nn$ along the other edges is implemented by approximation
(piecewise constant $L_2$-projection in $H^{-1/2}(\Gamma)$ and piecewise linear
interpolation in $H^{1/2}(\Gamma)$, with subsequent extensions to elements of the discrete
trace spaces).

\begin{figure}
\begin{center}
\scalebox{1.1}{\resizebox{0.4\textwidth}{0.5\textwidth}{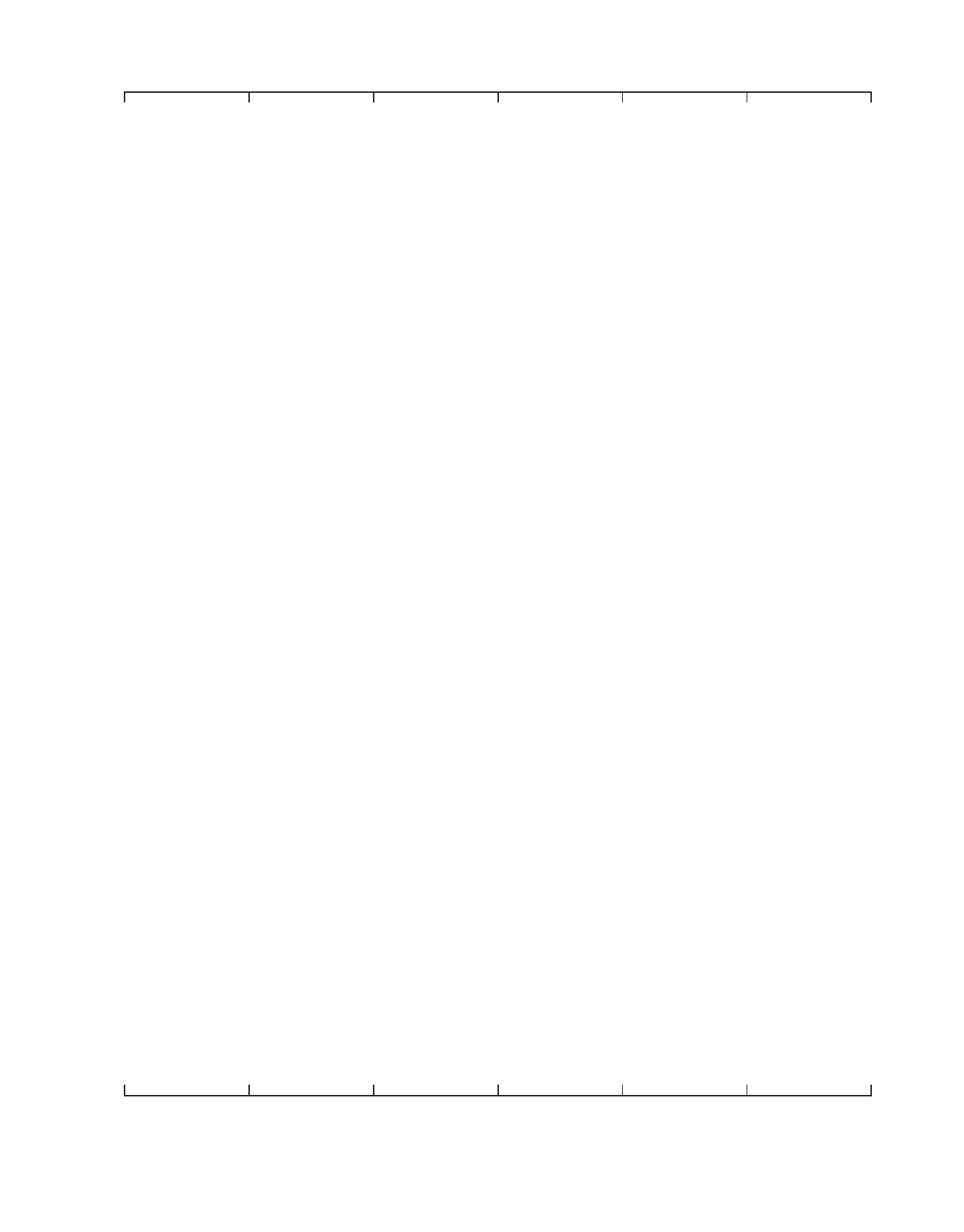}}
\scalebox{1.1}{\resizebox{0.4\textwidth}{0.5\textwidth}{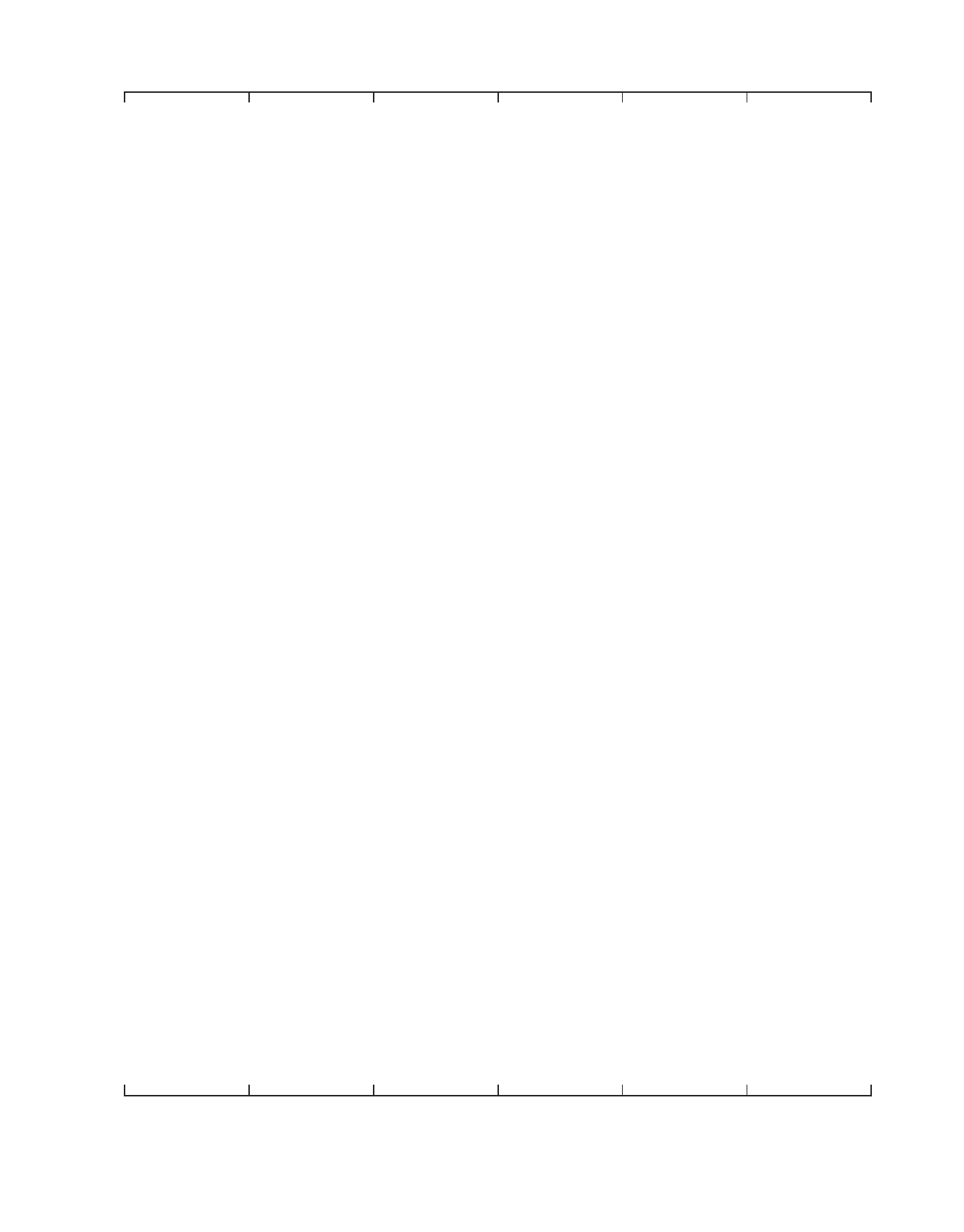}}
\end{center}
\caption{Initial mesh (left) and an adaptively refined mesh with 1674 elements}
\label{fig_mesh}
\end{figure}

When considering the first-order system \eqref{prob_sys1} we are dealing with solution
$\uone=\bu$ and zero components $\utwo,\uthree,\ufour$. In the case of the second-order
system \eqref{prob_sys2}, we have the solution $\bu$ and zero component $\bw$.
In any case, the first (and only) non-zero component (along with its canonical trace)
has a reduced regularity $\bu\in \bH^s(\Omega)$ with $s<2/3$, a standard Sobolev space.
Therefore, for a sequence of quasi-uniform meshes, we expect a reduced convergence
order $O(h^{2/3})=O(\dim(\UU_h)^{-1/3})$, whereas for a scheme with appropriate mesh refinement
at the incoming corner it should be of optimal order $O(\dim(\UU_h)^{-1/2})$.
The DPG method belongs to the family of minimum residual methods so that a posteriori
error estimation is inherent, cf.~\cite{CarstensenDG_14_PEC} for details.
We use the local residuals measured in $\VV_h(T)'$, $T\in\cT$,
to steer mesh refinements, based on newest-vertex-bisection and D\"orfler (or bulk)
marking with parameter $3/4$.
For both schemes, our numerical results confirm the reduced convergence order
when using quasi-uniform meshes, see Figure~\ref{fig_sing} with first scheme on the left and
second on the right, and illustrate optimal convergence of the adaptive variants,
see Figure~\ref{fig_sing_adap}. An adaptively refined mesh is shown in Figure~\ref{fig_mesh}.

\begin{figure}
\begin{center}
\includegraphics[height=0.4\textwidth]{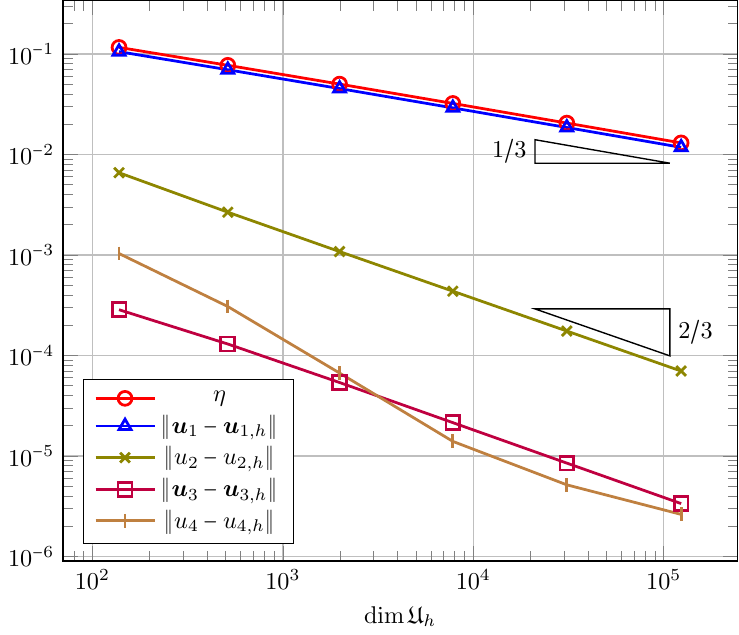}\quad
\includegraphics[height=0.4\textwidth]{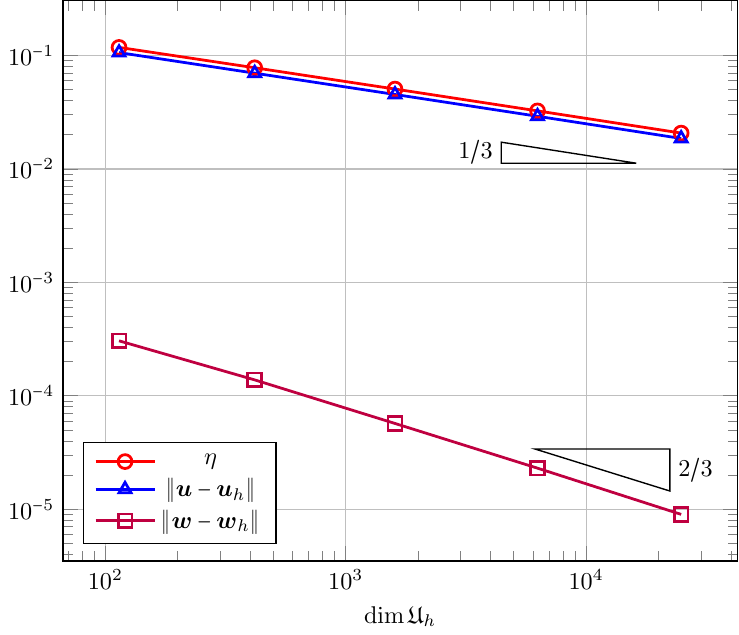}
\end{center}
\caption{Singular example, quasi-uniform meshes. Left: first-order system, right: second-order system}
\label{fig_sing}
\end{figure}

\begin{figure}
\begin{center}
\includegraphics[height=0.405\textwidth]{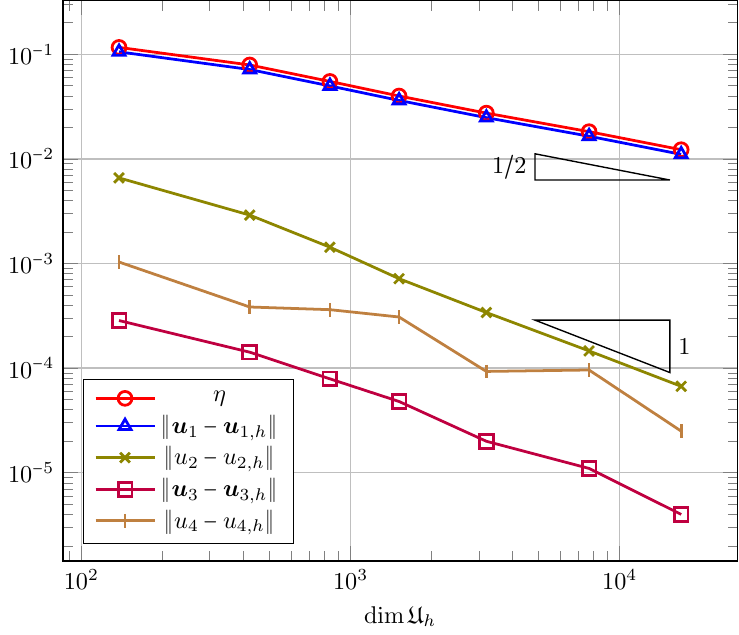}\quad
\includegraphics[height=0.4\textwidth]{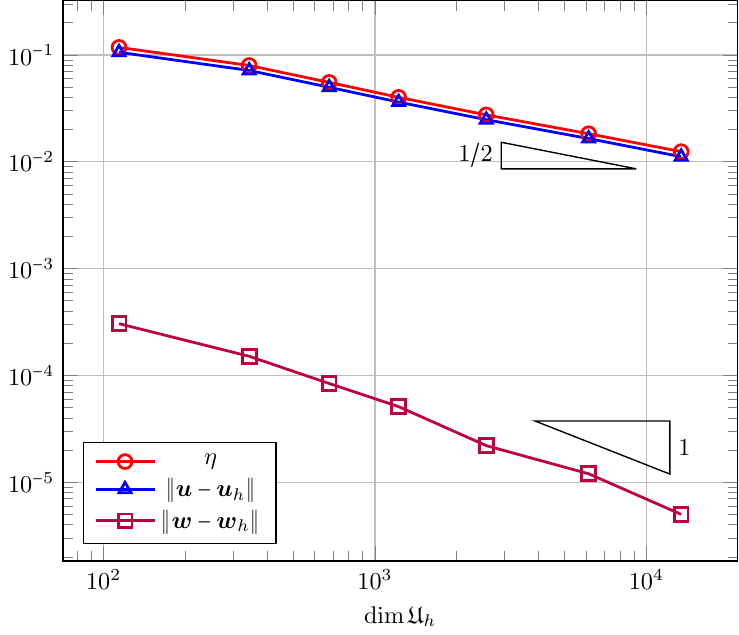}
\end{center}
\caption{Singular example, adaptively refined meshes. Left: first-order system, right: second-order system}
\label{fig_sing_adap}
\end{figure}


\bibliographystyle{siam}
\bibliography{}
\end{document}